\documentclass[11pt]{amsart}
\usepackage{amssymb,amsmath,amsthm,amsfonts,verbatim,tikz,mathrsfs,stmaryrd}
\SetSymbolFont{stmry}{bold}{U}{stmry}{m}{n}

\usepackage{mathtools}
\usepackage{lscape}
\usepackage{comment}
\usepackage[all,cmtip]{xy}
\input diagxy
\usepackage{hyperref}

\usepackage[margin=2.7cm]{geometry}
\usepackage[obeyFinal]{todonotes}
\usepackage{tikz-qtree}
\usepackage{bm}
\AtBeginDocument{\DeclareMathSymbol{:}{\mathpunct}{operators}{"3A}}

\hypersetup{
    colorlinks,
    linkcolor={red!50!black},
    citecolor={green!50!black},
    urlcolor={blue!80!black}
}

\numberwithin{equation}{section}
\newtheorem{theorem}[equation]{Theorem}
\newtheorem{thm}[equation]{Theorem}
\newtheorem{proposition}[equation]{Proposition}
\newtheorem{corollary}[equation]{Corollary}
\newtheorem{lemma}[equation]{Lemma}
\newtheorem{prop}[equation]{Proposition}

\theoremstyle{definition}

\newtheorem{definition}[equation]{Definition}
\newtheorem{notation}[equation]{Notation}

\newtheorem{defn}[equation]{Definition}

\theoremstyle{remark}
\newtheorem{remark}[equation]{Remark}
\newtheorem{rmk}[equation]{Remark}
\newtheorem{question}[equation]{Question}
\newtheorem{example}[equation]{Example}
\newtheorem{problem}[equation]{Problem}

\newcommand{\C}{\mathcal{C}}
\newcommand{\Q}{\mathbf{Q}}
\newcommand{\into}{\hookrightarrow}
\newcommand{\onto}{\twoheadrightarrow}
\newcommand{\Var}{\mathcal{V}}
\newcommand{\Varc}{\Var^{cptd}}

\newcommand{\F}{\mathbf{F}}
\newcommand{\Z}{\mathbf{Z}}
\newcommand{\N}{\mathbf{N}}
\newcommand{\mc}{\mathcal}
\newcommand{\Cb}{\mathbf{C}}
\newcommand{\Pb}{\mathbb{P}}
\newcommand{\Qb}{\mathbf{Q}}
\newcommand{\Rb}{\mathbf{R}}
\newcommand{\Zb}{\Z}
\newcommand{\Sb}{\mathbb{S}}
\DeclareMathOperator{\Chb}{Ch^b}
\DeclareMathOperator{\Chfb}{Ch^{fb}}
\DeclareMathOperator{\Chstr}{Ch^{str}}
\DeclareMathOperator{\cofib}{hocofib}

\newcommand{\op}{\text{op}}
\newcommand{\Fqbar}{\overline{\F}_q}
\newcommand{\et}{et}
\newcommand{\Frob}{Frob}
\newcommand{\MHS}{MHS}
\newcommand{\M}{\mathbf{M}}
\newcommand{\SB}{SB}
\newcommand{\rig}{rig}
\newcommand{\mbf}{\mathbf}

\newcommand{\rto}{\to}
\newcommand{\FinSet}{\mathbf{FinSet}}

\newcommand{\Oc}{\mathcal{O}}
\newcommand{\Zeta}{\boldsymbol{\zeta}}

\renewcommand{\tilde}{\widetilde}
\renewcommand{\bar}{\overline}

\newcommand{\rcofib}{\ar@{^{(}->}}
\newcommand{\lcofib}{\ar@{_{(}->}}
\newcommand{\compar}{\ar@{||->}}

\DeclareMathOperator{\Ar}{Ar}

\DeclareMathOperator{\co}{\mbf{co}}

\DeclareMathOperator{\comp}{\mbf{comp}}
\DeclareMathOperator{\w}{\mbf{w}}

\DeclareMathOperator{\Sh}{Sh}
\DeclareMathOperator{\Spec}{Spec}
\DeclareMathOperator{\Aut}{Aut}
\DeclareMathOperator{\End}{End}

\DeclareMathOperator{\Repc}{Rep_{cts}}
\DeclareMathOperator{\RepcAR}{Rep_{cts}^{AR}}
\DeclareMathOperator{\Gal}{Gal}
\DeclareMathOperator{\characteristic}{char}
\DeclareMathOperator{\Mod}{Mod}
\DeclareMathOperator{\Modf}{Mod^{fg}}
\DeclareMathOperator{\Image}{Image}
\DeclareMathOperator{\Pro}{Pro}
\DeclareMathOperator{\AR}{AR}
\DeclareMathOperator{\Sym}{Sym}
\DeclareMathOperator{\Ch}{Ch}
\DeclareMathOperator{\coker}{coker}
\DeclareMathOperator{\Fun}{Fun}
\DeclareMathOperator{\Ho}{Ho}
\DeclareMathOperator{\rk}{rk}
\newcommand{\CC}{\mathbf{C}}
\newcommand{\E}{\mathcal{E}}

\renewcommand{\emph}{\textsl}

\title{Derived $\ell$-adic Zeta Functions}
\author{Jonathan Campbell, Jesse Wolfson, and Inna Zakharevich}

\address{Department of Mathematics, Vanderbilt University}
\email{j.campbell@vanderbilt.edu}
\address{Department of Mathematics, University of California, Irvine}
\email{wolfson@uci.edu}
\address{Department of Mathematics, Cornell University}
\email{zakh@math.cornell.edu}

\begin{document}
\maketitle

\begin{abstract}
  Let $K_0 (\Var_k)$ be the Grothendieck group of $k$-varieties. Campbell and Zakharevich have constructed a higher algebraic $K$-theory spectrum $K(\Var_k)$ such that $\pi_0 K(\Var_k) = K_0 (\Var_k)$. In this paper we construct non-trivial classes in the higher homotopy groups of $K(\Var_k)$ when $k$ is finite or a subfield of $\CC$. To do this we give a recipe for lifting motivic measures $K_0 (\Var_k) \rto K_0 (\E)$ to maps of spectra $K(\Var_k) \rto K(\E)$.  We consider two special cases: the classical local zeta function, thought of as a homomorphism $K_0(\Var_{\F_q}) \rto K_0(\End(\Q_\ell))$, and the compactly-supported Euler characteristic, thought of as a homomorphism $K_0(\Var_{\CC}) \rto K_0(\Q)$.  We use lifts of these motivic measures to prove that the Grothendieck spectrum of varieties contains nontrivial geometric information in its higher homotopy groups by showing that the map $\Sb \rto K(\Var_k)$ is nontrivial in higher dimensions when $k$ is finite or a subfield of $\CC$, and, moreover, that when $k$ is finite this map is not surjective on higher homotopy groups.


\end{abstract}

\section{Introduction}
Many constructions in arithmetic and geometry give rise to ``motivic measures'', i.e. assignments
\[
    X\mapsto \mu(X)\in A
\]
which send an algebraic variety $X$ to $\mu(X)\in A$, for $A$ an abelian group, such that for any closed subvariety $Z\subset X$,
\[
    \mu(X)=\mu(X-Z)+\mu(Z).
\]
This observation led Grothendieck to introduce a universal motivic measure, via the ``Grothendieck ring of varieties'' $K_0(\Var_k)$:
\[
    K_0(\Var_k):=\Z[\mathrm{isoclass}(\Var_k)]/\{ [X]=[X-Z]+[Z]~|~Z\subset X~\text{closed}\}.
\]
where $\Var_k$ denotes the category of $k$-varieties, i.e. reduced separated
$k$-schemes of finite type, and where $\mathrm{isoclass}(\Var_k)$ denotes the
set of isomorphism classes of varieties.

The Grothendieck ring of varieties arises by forgetting large parts of the category $\Var_k$. Similarly, many motivic measures of interest arise from forgetting large parts of functorial constructions on varieties, including
\begin{enumerate}
\item the point counting measure $\#(-)$, which comes from applying
  $\mathrm{isoclass}$ to the functor from $\F_q$-varieties to finite sets
        \begin{align*}
            \Var_{\F_q}&\to \FinSet\\
            X&\mapsto X(\F_q)
        \end{align*}
      \item the compactly supported Euler characteristic $\chi(-)$, which comes
        from applying $K_0$ to the functor from the category $\Cb$-varieties and
        isomorphisms ($\Var_{\Cb}^\times$) to the homotopy category of
        homologically bounded rational cochain complexes
        \begin{align*}
            \Var_{\Cb}^{\times}&\to\Ho(\Chb(\Q))\\
            X&\mapsto C_c^\bullet(X(\Cb);\Q)
        \end{align*}
      \item the zeta function $\zeta_{(-)}(s)$,\footnote{While essentially
          classical, following from the Weil conjectures and work of Almkvist
          \cite{almkvist_3} (also explicated beautifully in Grayson
          \cite{Gr79}), this example does not appear to be sufficiently widely
          known.} which, by the Weil conjectures \cite{De}, comes from applying $K_0$ to the functor from the
        category of $\F_q$-varieties and isomorphisms ($\Var_{\F_q}^\times$)
        to the homotopy category of automorphisms of homologically bounded
        $\ell$-adic cochain complexes
        \begin{align*}
            \Var_{\F_q}^{\times}&\to\Ho(\Chb(\Aut(\Q_\ell)))\\
            X&\mapsto \Frob_q\circlearrowright R\Gamma_c(X_{\Fqbar};\Q_\ell)
        \end{align*}
        which sends a variety to the compactly supported $\ell$-adic cohomology of its restriction to $\Fqbar$ along with the action of Frobenius.
      \end{enumerate}
      (For a more in-depth discussion of the
Grothendieck ring of varieties, together with several other important example of
motivic measures, see \cite{nicaisesebag,hales,sahasrabudhe}.)

Recent work of the first and third authors (\cite{campbell, zakharevich, CZ-cgw}) constructs a higher Grothendieck ring of varieties $K(\Var_k)$, namely a commutative ring spectrum (in the sense of stable homotopy theory) with $\pi_0 K(\Var_k)\cong K_0(\Var_k)$. As with any higher algebraic $K$-theory, the structure of this spectrum is much richer than just its $\pi_0$. For example, the third author shows that $K(\Var_k)$ is tightly linked with birational geometry \cite{zakharevich_annihilator} and the first author shows that when $k = \CC$ it is related to Waldhausen's algebraic $K$-theory of spaces \cite{campbell}. On the other hand, despite these links, there are essentially no tools in the literature to {\em exhibit} nontrivial classes in $\pi_i(K(\Var_k))$ for $i>0$. Just as $K(\Var_k)$ represents a substantial refinement of the approximation to $\Var_k$ provided by the classical Grothendieck ring, one can ask for $K$-theoretic refinements of the approximations of important functors provided by classical motivic measures.

The purpose of the present article is to show that the classical functorial
constructions in the last two examples above\footnote{The first example is much
  simpler, and is already treated in \cite{campbell, zakharevich}.}  determine
such refinements, and to use them to exhibit the first examples of nontrivial
classes in the higher homotopy groups of $K(\Var_k)$.

More precisely, let $k$ be a field, $k^s$ a separable closure of $k$, and
$\ell\neq\characteristic(k)$ a prime. Let $\Chb(\Repc(\Gal(k^s/k);\Z_\ell))$
denote the category of homologically finite cochain complexes of continuous
integral $\ell$-adic $\Gal(k^s/k)$ representations. For a variety $X$, let
$[\Gal(k^s/k)\circlearrowright H_c^\bullet(X;\Z_\ell)]$ denote the compactly
supported $\ell$-adic cohomology of $X$, viewed as an object of the homotopy
category $\Ho(\Chb(\Repc(\Gal(k^s/k);\Z_\ell)))$ (see
Section~\ref{sec:prelim-coh} for precise definitions). Recall that a
``$W$-exact'' functor is the data required to induce a map on $K$-theory in this
context (see Section~\ref{sec:prelim} for the precise definition). We
prove:
\begin{theorem} \label{thm:main}
  In the notation above, the functor
    \begin{align*}
        \Var_{k}^{\times}&\to\Ho(\Chb(\Repc(\Gal(k^s/k);\Z_\ell)))^{\op}\\
            X&\mapsto [\Gal(k^s/k)\circlearrowright H_c^\bullet(X;\Z_\ell)]
  \end{align*}
  admits a strict model as a span of $W$-exact functors
  \begin{align*}
    \xymatrix{
            \Var_k & \widetilde{\Var_k} \ar[l]_{\sim} \ar[r] & \Chb(\Repc(\Gal(k^s/k);\Z_\ell))^{\op}}
  \end{align*}
  where the left arrow induces an equivalence on $K$-theory.
\end{theorem}

Now let $R$ be a ring, and let $W^{rat}(R)$ denote the ring of (big) rational Witt vectors of $R$. Work of Almkvist \cite{almkvist_3}, explicated beautifully in Grayson \cite{Gr79}, establishes the isomorphism
\[
    \xymatrix{
    K_0(\Aut(R))\ar[rr]_\cong^{\det(1-t\cdot f)} && W^{rat}(R)}
\]
where $\Aut(R)$ denotes the category of $R$-modules equipped with an automorphism. Now take $k=\F_q$ and $R=\Q_\ell$. Recall that $\Gal(\Fqbar/\F_q)$ is canonically freely topologically generated by Frobenius $\Frob_q$. Then specializing Theorem~\ref{thm:main} and applying $K$-theory, we obtain the following.
\begin{corollary} \label{cor:HWzeta} Let $k=\F_q$ be a finite field and
  $\ell \nmid q$ a prime. The strict model of Theorem~\ref{thm:main} determines a contractible space of maps of $K$-theory spectra
  \begin{equation*}
	   \Zeta: K(\Var_{\F_q})\to K(\Aut(\Z_\ell)),
    \end{equation*}
    which fits into a commuting square
    \begin{equation*}
          \xymatrix@C3em{
            K_0(\Var_{\F_q}) \ar[r]^-{\pi_0\Zeta} \ar[d]_{Z(-,t)} &
            K_0(\Aut(\Z_\ell)) \ar[d]^{\det(1-t\cdot Frob_q^\ast)}\\
            W^{rat}(\Zb) \ar[r] & W^{rat}(\Zb_\ell)}
	\end{equation*}
    after applying $\pi_0$.
\end{corollary}

The $K$-theory spectrum $K(\Var_k)$ can be thought of as a topological object that remembers not only that certain varieties decompose into other varieties, which is what $K_0 (\Var_k)$ does, but also \textit{how}. In the same way, the map in Thm~\ref{cor:HWzeta} tells us how the decomposition of varieties gets reflected in the decomposition of zeta functions. This is a great deal of information, and already at the level of the first homotopy group gives subtle information about automorphisms of varieties. We note below that this lift is canonical, and fully determined by the functor in Theorem~\ref{thm:main}.

We refer to $\Zeta$ as the \emph{derived $\ell$-adic zeta function}. More generally, for any field $k$ and any $g\in \Gal(k^s/k)$, we obtain (Corollary~\ref{cor:gZeta}) a ``derived $\ell$-adic $g$-Zeta function''
\[
    \Zeta_g\colon K(\Var_k)\to K(\Aut(\Z_\ell))
\]
which induces the map
\begin{align*}
    K_0(\Var_k)&\to W(\Z_\ell)\\
    [X]&\mapsto \sum_{i=0}^{2\dim X} (-1)^i \det(1-t\cdot g^\ast \circlearrowleft (H^i_c(X;\Z_\ell))).
\end{align*}

The method developed to prove Theorem~\ref{thm:main} can also be used to
construct other motivic measures.  For example, when $k$ is a subfield of $\Cb$
we can prove the following:
\begin{theorem}\label{cor:HWeuler}
    Let $k$ be a subfield of $\Cb$. Let $R$ be a ring. There exists a contractible space of maps of $K$-theory spectra
  \begin{equation*}
	   \mathbf{X}: K(\Var_{k})\to K(R),
    \end{equation*}
    which fits into a commuting square
    \begin{equation*}
          \xymatrix@C3em{
            K_0(\Var_{k}) \ar[r]^-{\pi_0 \mathbf{X}} \ar[d]^{\chi} &
            K_0(R) \ar[d]^{\rk}\\
            \Z \ar@{=}[r] & \Z}
	\end{equation*}
    after applying $\pi_0$.
\end{theorem}
We refer to $\mathbf{X}$ as the \emph{derived compactly supported Euler
  characteristic}.  A special case of this when $R = \Z_\ell$ or $\Z$ can be
obtained from Theorem~\ref{thm:main} by using the isomorphism between
$\ell$-adic and singular cohomology of complex varieties.  However, as the proof
of this theorem is significantly simpler than that of Theorem~\ref{thm:main}, we
present it separately (and, indeed, use it as an illustrative example of the
general method).  For more details and the proof, see
Sections~\ref{sec:approach} and \ref{sec:longproof}, in particular
Corollary~\ref{cor:Reuler}.

\begin{remark}
    While we state the theorem for homotopy categories for conciseness, our proof shows more. For readers who prefer the language of $\infty$-categories, we can summarize this as follows: the classical construction of compactly supported $\ell$-adic cohomology determines a pair of $\infty$-functors
    \begin{align*}
        (-)_!\colon \Var_k^{open}&\to \Chb(\Repc(\Gal(k^s/k);\Q_\ell))\\
        (-)^!\colon \Var_k^{closed,\op}&\to \Chb(\Repc(\Gal(k^s/k);\Q_\ell))
    \end{align*}
    where $\Var_k^{open}$ (resp. $\Var_k^{closed}$) denotes the subcategories of open (resp. closed) embeddings, and where the target denotes the $\infty$-category underlying the dg-category of homologically finite complexes of continuous $\ell$-adic Galois representations. Our proof shows that this pair of $\infty$-functors admits a {\em strict} representative as a span of $W$-exact functors of 1-categories as in Theorem~\ref{thm:main}. The span associates to each $k$-variety a filtering diagram in $\Chb(\Repc(\Gal(k^s/k);\Q_\ell)))$ in which all maps are weak equivalences. The pair of $\infty$-functors are obtained by taking a homotopy inverse limit over this (homotopically constant) diagram. Specializing to $k=\F_q$, we conclude that compactly supported $\ell$-adic cohomology determines a lift of the classical zeta function to a map of spectra
    \[
        K(\Var_{\F_q})\to K(\Aut(\Q_\ell))
    \]
    as in Corollary~\ref{cor:HWzeta}. Because compactly supported $\ell$-adic cohomology is classically defined as a homotopy inverse limit of right derived functors (i.e. a ``derived of a derived''), we refer to this map of spectra as a ``derived zeta function''. A similar, though simpler, discussion applies to Theorem~\ref{cor:HWeuler}.

    While the language of $\infty$-categories allows for useful conceptual statements, it frequently requires significant baggage if one wants to explicitly translate classical references into this framework. To cut down on this baggage,\footnote{e.g. Because categories with weak equivalences provide the only model for $\infty$-categories in which the first and third authors' extension of algebraic K-theory has been defined.} we have chosen to systematically eschew overt mention of $\infty$-categories and $\infty$-functors. However, the interested reader can easily verify the above claims in her favorite model of compactly supported $\ell$-adic cohomology as an $\infty$-functor.
\end{remark}

\begin{remark}
  One use of $\infty$-categories would be to prove a universality statement for our spectral lift of the zeta function. The most general statement of the universality of algebraic $K$-theory appears to be Barwick's \cite{barwick} phrased in terms of Waldhausen $\infty$-categories.  In order to make such a universality statement in our context, one would have to port this machinery to SW-categories (see \ref{sec:prelim} for definitions). While we are confident that this can be done, given the formal similarities between SW and Waldhausen categories, to do so would substantially lengthen the paper.

  There are, however, other ways to see an appropriate universality statement. Steimle \cite{steimle} has shown that the $K$-theory of Waldhausen categories is an initial functor in a certain category of Euler characteristics \cite[Thm. 0.2]{steimle}. Again, given the formal similarities between Waldhausen categories and SW-categories, it is clear that this statement can be made for SW-categories, and with significantly less technical upfront cost. However, we content ourselves with noting that such a universality statement can be made. The spectral lift of the zeta function we obtain is the (homotopy) initial map lifting the zeta function.
\end{remark}

Just as motivic measures allow one to probe the Grothendieck ring of varieties (and thus to coarsely probe the category of varieties itself), the derived motivic measures above allow us to probe the higher Grothendieck ring of varieties. Most ambitiously, one would like to understand the following:
\begin{question}\label{q:whatinfo}
    What arithmetic or geometric information do the higher homotopy groups $K_i(\Var_k)$ encode?
\end{question}
Given the complexity of higher algebraic K-theory in general, and given that several decades of effort have only begun to answer this question for $K_0(\Var_k)$, a full answer to Question~\ref{q:whatinfo} is likely a long ways off. One might instead begin by asking a milder question: are there even nonzero elements in $K_i (\Var_k)$ for $i > 0$? One simple example of an $SW$-category is $\FinSet$, the category of finite sets; by the theorem of Barratt-Priddy-Quillen \cite{barratt_priddy,segal_categories} the $K$-theory of this is equivalent to $\Sb$, the
sphere spectrum.  Note that for any variety $X$ we can define a map
\[\sigma_X: \Sb \rto K(\Var_k)\]
by thinking of $\Sb$ as $K(\FinSet)$ and sending the finite set $F$ to
$\coprod_F X$. When $k$ is a subfield of $\CC$ this gives enough information to detect some of
these nonzero elements.

\begin{theorem}\label{thm:subfield_of_C}
  Let $k$ be a subfield of $\CC$. Then there are arbitrarily high nonzero homotopy groups of $K(\Var_k)$.
\end{theorem}

The proof proceeds by tracing through the map $\pi_*\sigma_X$ and showing that for a fixed $s$ and appropriately chosen $X$, it is nonzero in degrees $4s-1$.  Thus in particular $K_{4s-1}(\Var_k)$ is
non-zero for all $s$. Using an elaboration of this proof could produce many more
non-trivial homotopy groups. For example, from the multiplicative structure of
$K(\Var_k)$ and the inclusion $i: \Aut(X) \to K_1(\Var_k)$, we have, for any
automorphism $f$ of $X$, the composite
\[
\pi_\ast \mathbb{S} \to K_\ast (\Var_k) \xrightarrow{\cdot i(f)} K_{\ast+1}(\Var_k) \to K_{\ast+1}(\Z).
\]
We are hopeful that this gives a method for detecting other non-trivial homotopy
groups of $K_\ast (\Var_k)$, but leave it to future work. We note that this is
an approach very similar to the one B\"{o}kstedt and Waldhausen
\cite{bokstedt_waldhausen} use to detect non-trivial homotopy groups in the
algebraic $K$-theory of spaces, $A(\ast)$.

When $k$ is finite we can obtain more refined information.  For a finite field
$k=\F_q$, the map $X\mapsto X(\F_q)$ defines a map $K(\Var_{\F_q})\to \Sb$ which
is a cosection of $\sigma_{\Spec k}$, yielding
$K(\Var_{\F_q}) \simeq \Sb \vee \widetilde{K}(\Var_{\F_q})$.  Thus the analog of
Theorem~\ref{thm:subfield_of_C} for $k$ finite is essentially trivial as it
exhibits $\pi_*\Sb$ is a summand of $K_*(\Var_k)$.  Defining
$\widetilde{K}(\Var_k) = \cofib \sigma_{\Spec k}$, a more subtle question asks:
\begin{question}
  Do there exist nontrivial elements in $\widetilde{K}_i (\Var_{k})$ for $i > 0$?
\end{question}
Note that this statement makes sense for any $k$. We use the derived zeta function to answer this question affirmatively.

\begin{theorem} \label{thm:calc} The group $\tilde K_1(\Var_k)$ is nontrivial
  whenever $k$ is a subfield of $\Rb$, a finite field with
  $|k| \equiv 3 \pmod 4$, or a global or local field with a place of cardinality
  $3\pmod 4$.
\end{theorem}

To prove this theorem we use the derived $2$-adic zeta function.  For any
category $\C$, let $\Aut(\C)$ be the category of pairs $(P,f)$, where $P\in \C$
and $f \in \Aut(P)$.  Morphisms $(P,f) \rto (Q,g)$ are morphisms $h:P \rto Q$
such that $hf = gh$.  When $\C$, is exact it induces an exact structure on
$\Aut(\C)$. In particular, for a ring $R$, we obtain an exact category $\Aut(R):=\Aut(\Modf(R))$. Specializing from a representation of $\Gal(\bar \F_q/\F_q)$ to its
value on $\Frob_q$ and tensoring by $\Qb$ gives a functor
$\Repc(\Gal(\bar \F_q/\F_q);\Z_\ell) \rto \Aut(\Qb_\ell)$.  Composing the
derived $2$-adic zeta function with the $K$-theory of this functor and applying $\pi_1$,
gives a homomorphism $K_1 (\Var_{k}) \to K_1 (g\Aut(\Qb_2))$.  The group
$K_1 (\Aut(\Qb_2))$ is relatively well-understood. In \cite{Gr79}, Grayson
constructs a homomorphism $\sigma_2: K_1 (\Aut(\Qb_2)) \to K_2 (\Qb_2)$. By
Moore's Theorem (see e.g. \cite[Appendix]{milnor}), the 2-adic Hilbert symbol
induces a (split) surjection $(-,-)_2\colon K_2(\Qb_2)\onto\Z/2\Z$, and by
composing these maps, we produce a map
\[
h_2\colon K_1 (\Var_{k}) \to K_1 (\Aut(\Qb_2)) \to^{\sigma_2} K_2 (\Qb_2) \to^{(-,-)_2} \Z/2\Z.
\]
We then show $h_2 \circ \pi_1\sigma_{\Spec k}$ is trivial, but $h_2 \circ
\pi_1\sigma_{\Pb^1}$ is surjective.

Corollary~\ref{cor:HWzeta} and Theorem~\ref{thm:calc} are both applications of
Theorem~\ref{thm:main}, in which $\Zeta$ is constructed.  In order to construct
$\Zeta$, we use a $K$-theory machinery first created by the first author in
\cite{campbell}. The usual categories one wants to work with as inputs for a
$K$-theory machine are Waldhausen categories \cite{waldhausen}. Unfortunately,
these do not work to produce $K(\Var_k)$, which is why the first and third
authors introduced their formalisms. In \cite{campbell}, the difficulty is
circumvented by defining a modification of Waldhausen categories called
$SW$-categories (the S is for ``scissors'') where one can define algebraic
$K$-theory for $\Var_k$ in much the same way one does for Waldhausen
categories. However, in order to get maps $K(\C) \to K(\mathcal{W})$ where $\C$
is an $SW$-category and $\mathcal{W}$ is a Waldhausen category, one needs the
notion of a ``$W$-exact functor'' introduced in \cite{campbell}. It needs to
satisfy certain variance conditions reminiscent of push-pull formulae (see
Section \ref{sec:prelim} for details). To construct the derived $\ell$-adic zeta
function, we take the $SW$-category $\Var_k$ and the Waldhausen category
$\Chb(\Repc(\Gal(k^s/k);\Z_\ell))$ of homologically finite and bounded chain
complexes of continuous $\ell$-adic Galois representations. To go from one to the other we need to use
compactly-supported \'etale cohomology. Classically, this is defined by choosing a compactification $j\colon X\to \bar{X}$, right deriving the functor $\Gamma\circ j_!$ and then taking a homotopy inverse limit over all such choices to get a well-defined object.

The key technical observation of this paper is that all of this can be made ``strict''.  We show that once we restrict to constant sheaves, the Godement resolution combines both the necessary functoriality and exactness properties required to produce a $W$-exact functor. By further replacing the homotopy inverse limit by the left-facing map in the span, we obtain the span of $W$-exact functors in Theorem~\ref{thm:main}.  Using invariance properties of $K$-theory, we can then invert the left-facing arrow and obtain the derived $\ell$-adic zeta function.

We view this work as part of a larger program to lift motivic measures to the
spectral/homotopical level.  For example, the outline we follow should adapt to give lifts for other cohomologically defined motivic measures, e.g. $p$-adic zeta functions, Serre polynomials, and the Gillet--Soul\'e measure \cite{GS}. One might similarly ask for lifts of Kapranov's motivic zeta function, or of the motivic measure used by Larsen and Lunts \cite{LL} to show that motivic zeta function is \emph{not} rational as a map out of $K_0(\Var_{\Cb})$. See Section~\ref{sec:prospects} for a more detailed discussion.

This paper is organized as follows. In Section \ref{sec:prelim}, we quickly review Waldhausen $K$-theory and introduce SW-categories.  In  Section~\ref{sec:prelim-coh}, we review the background and necessary results  from $\ell$-adic cohomology and Galois representations. In Section ~\ref{sec:approach} we review our general approach to constructing derived motivic measures, and provide the the construction of the motivic measure arising from singular cohomology.  Section~\ref{sec:longproof} contains the full construction of the derived $\ell$-adic zeta function. In Section~\ref{sec:nontriv} we use the results of the previous section to construct nontrivial elements in the higher K-theory of varieties over both $\CC$ and finite fields. We close, in Section~\ref{sec:prospects}, by discussing questions for future work.

\subsection*{Acknowledgments}
We thank Bhargav Bhatt, Denis-Charles Cisinski, Sean Howe, Keerthi Madapusi Pera and Nick Rozenblyum for helpful correspondence. We thank Oliver Braunling, Kiran Kedlaya, Dan Petersen, Ravi Vakil, Chuck Weibel, Kirsten Wickelgren and Ilya Zakharevich for many helpful questions and comments on an earlier draft. We thank the anonymous referee for a careful reading and many helpful comments which greatly improved the paper. J.W. was supported in part by NSF Grant No. DMS-1400349.  I.Z. was supported in part by an NSF MSPRF grant and NSF Grant No. DMS-1654522.

\begin{notation}
  Throughout, when dealing with schemes or varieties, we let
  $Z \hookrightarrow Y$ denote a closed inclusion and $ X \xrightarrow{\circ} Y$
  denote an open inclusion.
\end{notation}

\section{$SW$-categories and $K$-Theory} \label{sec:prelim}

In \cite{zakharevich}, the third author defines a spectrum $K(\Var_k)$ whose
zeroth homotopy group is the Grothendieck ring of varieties over $k$. In \cite{campbell},
the first author gives an alternate construction of this spectrum. In this
paper, we use the latter construction to produce maps out of $K(\Var_k)$, so we
review the structure necessary to produce this spectrum.

Most definitions of $K$-theory work with categories where a suitable notion of
quotient exists, for example Quillen's exact categories \cite{quillen} or
Waldhausen's categories \cite{waldhausen}. These notions of quotient are then
used to define the exact sequences that $K$-theory is defined to ``split.'' When
dealing with the category of varieties, we have no such quotients. Instead, our
``exact sequences'' are sequences of the form $Z \hookrightarrow X
\xleftarrow{\circ} (X - Z)$ where the first map is a closed inclusion and the
second is an open inclusion. The notion of an $SW$-category is meant to modify
Waldhausen's definition of categories with cofibrations and weak equivalences to
allow the use of such ``exact sequences.''  For ease of reading, we review
Waldhausen's construction before recalling the first author's construction.

\begin{defn}[{Waldhausen category, \cite[Section 1.2]{waldhausen}}]
  A \emph{Waldhausen category}\footnote{Referred to as a ``category with
    cofibrations and weak equivalences'' by Waldhausen.} is a category $\mc{C}$
  equipped with two distinguished subcategories: cofibrations and weak
  equivalences, denoted $\mathbf{co}(\mc{C})$ and $\mathbf{w}(\mc{C})$.  The
  arrows in $\co(\mc{C})$ are denoted by hooked arrows
  $\hookrightarrow$. Arrows representing weak equivalences are decorated
  with $\sim$. These categories satisfy the following axioms:
  \begin{enumerate}
  \item $\mc{C}$ has a zero object $0$.
  \item All isomorphisms are contained in $\mathbf{co}(\mc{C})$ and $\mathbf{w}(\C)$.
  \item For all objects $A$ of $\C$, the morphism $0 \rto A$ is a
    cofibration.
  \item (\textbf{pushouts}) For any diagram
    \[\xymatrix{C & A \ar[l] \ar@{^{(}->}[r] & B}\]
    where $A\hookrightarrow B$ is a cofibration, the pushout exists and the morphism $C
    \hookrightarrow B \cup_A C$ is a cofibration.
  \item (\textbf{gluing}) For any diagram
    \[\xymatrix{ C \ar[d]^\sim & A \ar@{^{(}->}[r] \ar[l] \ar[d]^\sim & B
      \ar[d]^\sim \\
      C' & A' \ar@{^{(}->}[r] \ar[l] & B'}\]
    where the vertical morphisms are weak equivalences the induced morphism
    \[B\cup_A C \rto^{\sim} B' \cup_{A'} C'\]
    is also a weak equivalence.
  \end{enumerate}
\end{defn}

We will be using a slightly more general definition of the $K$-theory of a
Waldhausen category, as we need this flexibility for one of our main results.  

We begin with some preliminary definitions.

\begin{defn}
  Let $\mc{C}$ be a category.  The category $\Ar\C$ has, as its objects,
  the morphisms of $\mc{C}$.  The morphisms in $\Ar\C$ from a morphism $f:A \to
  B$ to a morphism $g:C \to D$ are commutative squares
  \[\xymatrix{
      A \ar[r] \ar[d]_f & C \ar[d]^g \\ B \ar[r] & D}.\]
  If $\mc{C}$ is equipped with a subcategory of weak equivalences then a
  morphism in $\Ar\C$ is considered a weak equivalence if both horizontal
  morphisms in the diagram above are weak equivalences.

  Let $[n]$ be the ordered set $\{0<\ldots<n\}$ considered as a category.  Then
  $\Ar[n]$ can be considered to be the set of pairs $(i,j)\in [n]\times[n]$ with
  $i \leq j$.  We are often considering functors $X: \Ar[n] \rto \mc{C}$; in
  this case, we write $X_{i,j}$ for $X(i,j)$.
\end{defn}

\begin{defn}
  Let $\mc{C}$ be a Waldhausen category.  A morphism in $\Ar\mc{C}$ is a
  \emph{weak cofibration} if it is weakly equivalent (via a zigzag of weak
  equivalences) in $\Ar\C$ to a cofibration.  A square in $\C$ is
  \emph{homotopy cocartesian} if it is weakly equivalent (by a zigzag) to a
  pushout square where either both horizontal or both vertical morphisms are
  cofibrations.
\end{defn}

\begin{defn}[{$S'_\bullet$-construction, see \cite[Definition 2.3]{BM-abstract}}]
  Let $\mc{C}$ be a Waldausen category. We define $S'_n \mc{C}$ to be the
  category of functors
  \[
  X: \Ar[n] \to \mc{C}
  \]
  with morphisms natural transformations, subject to the conditions
  \begin{itemize}
  \item the initial map $0 \rto A_{i,i}$ is a weak equivalence for all i.
  \item When $i \leq j \leq k$, $X_{i, j} \to X_{i, k}$ is a weak cofibration.
  \item For any $i \leq j \leq k$ the square
    \[\xymatrix{
    X_{i, j} \ar[r] \ar[d] &  X_{i, k} \ar[d]  \\
    0 \ar[r] & X_{j, k}}
    \]
    is a homotopy cocartesian square.
  \end{itemize}
  A map $A \to B$ in $S'_n\C$ is a weak equivalence when each component map
  $A_{i,j} \to B_{i,j}$ is a weak equivalence; it is a cofibration when each
  component map $A_{i,j} \to B_{i,j}$ is a cofibration and the map $A_{i,k}
  \cup_{A_{i,j}} B_{i,j} \to B_{i,k}$ is a weak cofibration.
\end{defn}
\begin{rmk}
  The $S'_n \mc{C}$ assemble to form a simplicial category (i.e. a simplicial object
  in the category of small categories).  For more detail on this, see
  \cite[Section 2]{BM-abstract}.
\end{rmk}

We now define the algebraic $K$-theory spectrum of a Waldhausen category.
Unfortunately, the $S'_\bullet$-construction does not work correctly for all
Waldhausen categories, but only those satisfying a condition Blumberg--Mandell
call \emph{functorial factorization of weak cofibrations (FFWC)}.  All examples
that we are concerned with satisfy this condition, but a discussion of the
condition is not illuminating for the sake of the current discussion; we prove
that this is the case in Appendix~\ref{app:FFWC} and restrict our attention to
categories satisfying FFWC.

\begin{defn}
  Let $\mc{C}$ be a Waldhausen category satisfying FFWC.  Let $w S'_n \mc{C}$
  denote the subcategory of weak equivalences of $S'_n \mc{C}$ and let
  $N_\bullet wS'_n \mc{C}$ denote the nerve of that category. The topological
  space $K^t(\mc{C})$ is defined by
  \[
  K^t(\mc{C}) = \Omega |N_\bullet wS'_\bullet \mc{C}|
  \]
  where $|-|$ denotes the geometric realization of a bisimplicial set. The spectrum
  $K(\mc{C})$ is defined by taking a (functorial) fibrant-cofibrant replacement
  in the stable model category of symmetric spectra \cite[Sec. 9]{mandell_may_schwede_shipley} of the spectrum whose $m$-th space is
  \[|N_\bullet w\underbrace{S'_\bullet\cdots S'_\bullet}_{m\ \mathrm{times}}
  \mc{C}|.\]
\end{defn}

The most important example of a Waldhausen category for the purposes of this
paper is the following:

\begin{example} \label{ex:chains}
  Let $\mc{E}$ be any exact category.  If we define the admissible monomorphisms
  to be the cofibrations and the isomorphisms to be the weak equivalences then
  the Waldhausen $K$-theory of $\mc{E}$ and Quillen's $K$-theory of $\mc{E}$ are
  equivalent.  Let $\Chstr(\mc{E})$ be the category of bounded chain complexes
  in $\mc{E}$; by  \cite[Theorem 1.11.7]{TT}, the incluson $\mc{E} \rto
  \Chstr(\mc{E})$ given by mapping $\mc{E}$ to the chain complexes concentrated
  at $0$ is an equivalence on $K$-theory.  The inverse map on $K_0$ is exactly
  the Euler characteristic.
\end{example}

  However, we need a stronger version of this. The following lemma is probably well-known to experts, but we could not find a statement in the literature. We have included it since it is needed below.

  \begin{lemma}\label{lem:chains}
    Let $R$ be a ring, and suppose
  that $\mc{E}$ is $\Modf_R$, the category of finitely generated $R$-modules.
  Let $\Chb(R)$ be the category of chain complexes of (possibly infinitely
  generated) $R$-modules whose cohomology is bounded and in $\Modf_R$.  The induced map \[
  K(\Modf_R) \rto K(\Chb(R))
  \]
  is an equivalence.
  \end{lemma}

  \begin{proof}
  We prove this by breaking the map up into three compositions:
  \[K(\Modf_R) \rto K(\Chstr(R)) \rto K(\Chfb(R)) \rto K(\Chb(R)).\] Here,
  $\Chfb(R)$ is the category of chain complexes of
  finitely generated $R$-modules which are homologically bounded.  The first of these is an equivalence by
  \cite[Theorem 1.11.7]{TT}.  The second is an equivalence by \cite[Section
  V.2.7.1]{weibel_kbook}.  Thus it remains to consider the third.  We prove that
  this is an equivalence by using Waldhausen's Approximation Theorem,
  \cite[Theorem 1.6.7]{waldhausen}.

  To apply the theorem we must show that for any map $f:A \to B$, where
  $A \in \Chfb(R)$ and $B\in \Chb(R)$ there exists a cofibration
  $g:A \hookrightarrow A'$ in $\Chfb(R)$ and a weak equivalence $f':A' \rto B$
  such that $f = f'g$.  Note that (by possibly first factoring $f$ as a
  cofibration followed by a weak equivalence) it suffices to check this when $f$
  is itself a cofibration; in particular $f$ is levelwise injective.  Suppose
  that the cohomology of $B$ is only nonzero below dimension $k$.  We define
  $A'_m = 0$ if $m > k$.  For $m \leq k$ we define $A_m'$ to be the submodule of
  $B_m$ generated by
  \begin{itemize}
  \item the image of $A_m$,
  \item a choice of generators for $H^m(B)$, and
  \item a choice of generators for the relations between the generators for
    $H^{m+1}(B)$.
  \end{itemize}
  (Since the cohomology is bounded we can just construct these starting at $m=k$
  and working downwards.)  Since each of these only involves a finite number of
  generators, $A'_m$ is finitely generated.  Thus $A'$ is levelwise a subcomplex
  of $B$ which has the same cohomology and satisfies the desired conditions.
  Thus the map $K(\Chfb(R)) \rto K(\Chb(R))$ is a weak equivalence.

  It is not necessarily the case that for all exact categories, $\Chb(\mc{E})$
  satisfies FFWC.  In Appendix~\ref{app:FFWC} we show that the specific cases we
  are interested in in this paper do; the results of this appendix also suggest
  that for almost all cases of $\mc{E}$ that are of interest FFWC holds.
\end{proof}

We now turn to defining $SW$-categories.  As much of the intuition necessary for
working with these comes from Waldhausen's $S_\bullet$-construction we omit the full definition and instead refer to \cite{campbell}.  Note also, we omit the ``weakness''
hypotheses, since we need to work inside $SW$-categories more strictly than
in Waldhausen categories.

\begin{defn}[{$SW$-category \cite[Definition 3.23]{campbell}}]
  An \emph{$SW$-category} is a category $\mc{C}$ equipped with three
  distinguished subcategories: cofibrations, complement maps, and weak
  equivalences, denoted $\co(\mc{C})$, $\mathbf{comp}(\mc{C})$ and
  $\w(\mc{C})$. The arrows in $\co(\mc{C})$ are denoted by hooked arrows
  $\hookrightarrow$ and the arrows in $\comp(\mc{C})$ are denoted by $\xrightarrow{\circ}$. Arrows representing weak equivalences are decorated with $\sim$. The category $\mc{C}$ is further equipped with a collection of \textit{subtraction sequences} $\{Z \hookrightarrow X \xleftarrow{\circ} U\}$. The data is required to satisfy axioms spelled out in \cite[Definition 3.7, Definition 3.13, Definition 3.24]{campbell} which mimic those of Waldhausen, replacing cofiber sequences with subtraction sequences.
\end{defn}

There are many examples of these kinds of categories. The following is the motivating example.

\begin{example}
  The category $\Var_{k}$ of varieties over a field $k$, is an $SW$-category,
  where cofibrations are closed immersions, complements are open immersions, and
  the weak equivalences are isomorphisms.  (This is proven in detail in the results leading up to
  \cite[Prop. 3.28]{campbell}) The subtraction
  sequences are defined as follows.  Given a closed inclusion $i: Z \to X$, $i$
  determines a homeomorphism of $Z$ onto a closed set $i(Z)$. We consider the
  open set $X - i(Z)$ and give it a scheme structure by restricting the
  structure sheaf on $X$. Thus
  \[
    X - Z = (X - i(Z), \mc{O}_{X}|_{X-Z}).
  \]
\end{example}

The definition of an $SW$-category is designed to provide exactly the structure
needed to carry out a Waldhausen-style $S_\bullet$-construction when we have
subtraction instead of quotients. However, we need one auxiliary definition.

\begin{defn}
   We define $\widetilde{\Ar}[n]$ to be the full subcategory of
  $[n]^{\text{op}} \times [n]$ consisting of pairs $(i, j)$ with $i \leq j$.
\end{defn}

\begin{defn}[$\widetilde{S}_\bullet$-construction]
  Let $\mc{C}$ be an $SW$-category. We define $\widetilde{S}_n \mc{C}$ to be the category with objects functors
  \[
  X: \widetilde{\Ar}[n] \to \mc{C}
  \]
  with morphisms natural transformations, subject to the conditions
  \begin{itemize}
  \item $X_{i, i} = \emptyset$, the initial object
  \item When $j < k$, $X_{i, j} \to X_{i, k}$ is a cofibration.
  \item The subdiagram
    \[
    X_{i, j} \longrightarrow X_{i, k} \longleftarrow X_{j, k}
    \]
    is a subtraction sequence for all $i < j < k$.
  \end{itemize}
\end{defn}
\begin{rmk}
  The $\tilde S_n \mc{C}$ assemble to form a simplicial category (i.e. a
  simplicial object in the category of small categories).  Each of these is
  itself an $SW$-category, so this construction can be iterated.  For details,
  see \cite[Lem.~3.34]{campbell}.
\end{rmk}

We may finally define the algebraic $K$-theory spectrum of an
$SW$-category.

\begin{defn}
  Let $\mc{C}$ be an $SW$-category.  Let $w \widetilde{S}_n \mc{C}$ denote the
  subcategory of weak equivalence of $\widetilde{S}_n \mc{C}$ and let
  $N_\bullet w\widetilde{S}_n \mc{C}$ denote the nerve of that category. The
  topological space $K^t(\mc{C})$ is defined by
  \[
  K^t(\mc{C}) = \Omega |N_\bullet w\widetilde{S}_\bullet \mc{C}|
  \]
  where $|-|$ denotes the geometric realization of a bisimplicial set. The spectrum
  $K(\mc{C})$ is defined by taking a (functorial) fibrant-cofibrant replacement
  in the stable model category of symmetric spectra of the spectrum whose $m$-th space is
  \[|N_\bullet w\underbrace{\tilde S_\bullet\cdots \tilde S_\bullet}_{m\ \mathrm{times}}
  \mc{C}|.\]
\end{defn}

There is a notion of exact functors for   $SW$-categories:
\begin{defn}
  Let $\mc{C}, \mc{D}$ be $SW$-categories. A functor $F: \mc{C} \to \mc{D}$ is called \emph{exact} if
  \begin{enumerate}
  \item $F$ preserves the initial object: $F(\emptyset) = \emptyset$.
  \item $F$ preserves subtraction sequences
  \item $F$ preserves pushout diagrams.
  \end{enumerate}
\end{defn}

\begin{proposition}\label{exact_map}
  Let $\mc{C}$ and $\mc{D}$ be $SW$-categories and let $F: \mc{C} \to \mc{D}$ be
  an exact functor. Then $F$ descends to a map of spectra $K(\mc{C}) \to
  K(\mc{D})$.
\end{proposition}
\begin{proof}
This follows directly from the definition of exact functor.
\end{proof}

Most of the maps in which we are interested do not have $SW$-categories as
codomains; instead, we wish to be able to construct a functor from an
$SW$-category to a Waldhausen category.  This requires we use a different
definition in order to define the map of $K$-theories, since we cannot just hit
the source and target with the $\widetilde{S}_\bullet$ construction or
$S'_\bullet$-construction. In fact, because of the change in variance, the
proper notion is not a functor at all --- instead it is a triple of functors,
two covariant (for weak equivalences and cofibrations) and one contravariant
(for the complement maps).  One should keep in mind here the dual of compactly
supported cohomology, which is covariant on closed inclusions and contravariant
on open.

\begin{defn}[Based on {\cite[Defn. 5.2]{campbell}}]\label{pseudo_exact}
  Let $\mc{C}$ be an $SW$-category and $\mc{D}$ a Waldhausen category. A
  \emph{$W$-exact functor} (resp. {\em weakly $W$-exact functor}) from $\mc{C}$ to $\mc{D}$ is a triple of
  functors $(F_!, F^!, F^w)$ such that
  \begin{enumerate}
  \item $F_!$ is a functor $F_!: \co(\mc{C}) \to \mc{D}$ from the
    category of cofibrations in $\mc{C}$ to $\mc{D}$. For a map $i$ we
    abbreviate $F_! (i)$ to $i_!$.
  \item $F^!$ a contravariant functor $F^!: \comp(\mc{C})^{\text{op}} \to \mc{D}$,
    from the category of complement maps in $\mc{C}$ to $\mc{D}$. For a map $j$ we
    abbreviate $F^! (j)$ to $j^!$.
  \item $F^w$ is a functor $\w(\mc{C}) \to \w(\mc{D})$.
  \item For objects $X \in \mc{C}$, $F^! (X) = F_! (X) = F^w (X)$ and we denote all three by $F(X)$.
  \item For every \emph{cartesian} diagram in $\mc{C}$ on the left below, where the horizontal maps are cofibrations and the vertical maps are complements, we obtain a commuting diagram on the right
    \[
    \xymatrix{
    X \ar@{^{(}->}[r]^{j}\ar[d]_{i} & Z\ar[d]^{i'} \\
    Y \ar@{^{(}->}[r]_{j'}  & W
    }
    \qquad
        \xymatrix{
      F(X) \ar[r]^{j_!} & F(Z)  \\
      F(Y) \ar[u]^{i^!} \ar[r]^{j'_!} & F(W)  \ar[u]_{(i')^!}
    }
    \]

  \item For a subtraction sequence in $\mc{C}$ on the left, the square in $\mc{D}$ on the right is cocartesian (resp. weakly cocartesian).
     \[
    \xymatrix{
      Z \ar@{^{(}->}[r]^i  & X\\
       & X - Z \ar[u]^{\circ}_-j
    }
    \qquad
          \xymatrix{
        F(X) \ar[r]^{i_!} \ar[d] & F(Y) \ar[d]^{j^!} \\
        F(0) \ar[r] & F(Y-X)
      }
    \]

  \item For any commutative diagram on the left below, where the horizontal morphisms are cofibrations and the vertical morphisms are weak equivalences, the diagram on the right commutes
    \[
    \xymatrix{
      X \rcofib[r]^f\ar[d]_{i_X} & Y \ar[d]^{i_Y}\\
      X' \rcofib[r]_{f'} & Y'
    }
    \qquad
        \xymatrix{
      F(X) \ar[r]^{f_!} \ar[d]_{F^w(i_X)} & F(Y) \ar[d]^{F^w(i_Y)} \\
      F(X') \ar[r]_{(f')_!}  & F(Y')
        }
    \]
    A similar statement holds for complement maps.
  \end{enumerate}
\end{defn}

\begin{rmk}
  For notational ease, we denote $W$-exact functors by
  $(F_!, F^!, F^w): \mc{C} \to \mc{D}$ or even $F:\mc{C} \rto \mc{D}$, when no
  confusion can arise.
\end{rmk}

Having defined this, one can prove the following.

\begin{proposition}[Based on {\cite[Prop. 5.3]{campbell}}]
  Let $\mc{C}$ be an $SW$-category and $\mc{D}$ a Waldhausen category with FFWC
  and let $(F_!, F^!,F^w ): \mc{C} \to \mc{D}$ be a weakly $W$-exact functor. Then
  $(F_!, F^!,F^w )$ determines a spectrum map
  \[
  K(\mc{C}) \to^F K(\mc{D}).
  \]
\end{proposition}

\begin{proof}
  It suffices to prove that this map exists before taking fibrant-cofibrant
  replacement, since our replacement is functorial.  But this follows exactly
  from the definitions of $K(\mc{C})$ and $K(\mc{D})$, since a weakly $W$-exact
  functor takes a simplex in the spaces defining $K(\mc{C})$ to a simplex in the
  spaces defining $K(\mc{D})$.
\end{proof}

As a consequence of the definition of $K$-theory, we obtain the following
result, which can be used to pick out interesting elements of $K_1$ of an
$SW$-category.  The proof is just as for Waldhausen $K$-theory as in \cite[\S
1.5]{waldhausen}.

\begin{proposition} \label{prop:xiX} Let $X$ be an object in an $SW$-category
  (resp. Waldhausen category) $\mc{C}$.  There is a homomorphism
  $\xi_X: \Aut(X) \rto K_1(\C)$, which is natural in $\C$ in the sense that for
  any exact (resp. weakly $W$-exact)
  functor $F: \C \rto \mc{D}$, $\xi_{F(X)}=\pi_1 F\circ\xi_X$.
\end{proposition}
\begin{proof}
  We begin by recalling how this statement works for Waldhausen categories (see
  \cite[p.341]{waldhausen}).  Given an automorphism $f: X \to X$ there is a
  corresponding 1-simplex in $w \mc{C}$. Composing with the map
  $|w \mc{C}| \to \Omega|w S_\bullet \mc{C}|$ gives the desired map
  $\Aut(X) \rto K_1 \mc{C}$. This construction is visibly natural for exact
  functors $F: \mc{C} \to \mc{D}$, and uses only that automorphisms include into
  $\pi_1 |w \mc{C}|$ and that we have a natural map $|w\mc{C}| \to |w S'_\bullet
  \mc{C}|$. As the same is true for the $S'_\bullet$-construction, an analogous
  statement holds.

  For SW-categories $\mc{C}$, the construction works the same way since the simplicial set $w\widetilde{S}_\bullet \mc{C}$ is still reduced and isomorphisms are a subcategory of weak equivalences.

  For a weakly W-exact functor, $F = (F_!, F^!, F^w)$ from an SW-category $\mc{C}$ to a Waldhausen category $\mc{D}$ we obtain a simplicial map $w \widetilde{S}_\bullet \mc{C} \to w \widetilde{S}'_\bullet \mc{D}$. For $X \in \mc{C}$, there are maps $\operatorname{Aut}(X) \to \operatorname{Aut} (F(X))$, and $|w\mc{C}| \to |w \mc{D}|$ induced by $F^w$. It is then clear that the diagram
  \[
  \xymatrix{
    \operatorname{Aut} (X) \ar[r]\ar[d] & \pi_1 w \mc{C} \ar[r] \ar[d] & \pi_1 (\Omega |w \widetilde{S}_\bullet \mc{C}|)\ar[d]\\
    \operatorname{Aut} F(X) \ar[r] & \pi_1 w \mc{D} \ar[r] &  \pi_1 (\Omega |w \widetilde{S}'_\bullet \mc{D}|)
  }
  \]
commutes, which is the statement of the proposition.

\end{proof}

An important tool in the general method discussed in
Section~\ref{sec:approach} is a lemma designed to identify when the
$K$-theories of two $SW$-categories are equivalent.  Although the conditions of
this lemma look complicated, in the geometric cases we are interested in they
are generally very natural.  For a detailed example of an application, see
Example~\ref{ex:compactification}.

\begin{lemma}\label{lem:Kequiv}
  Let $U: \mc{A} \rto \mc{C}$ be an exact functor of $SW$-categories.  Suppose
  that the following extra conditions hold:
  \begin{enumerate}
  \item $U$ is surjective on objects.  Moreover, one of the following holds:
    \begin{enumerate}
    \item For all cofibrations $f:X \hookrightarrow X'$, if $U(A') = X'$ then
      there exists a cofibration $A \hookrightarrow A'$ whose image under $U$ is
      $f$.
    \item For all cofibrations $f:X \hookrightarrow X'$, if $U(A) = X$ then
      there exists a cofibration $A \hookrightarrow A'$ whose image under $U$ is
      $f$.
    \end{enumerate}
  \item If for $f,g:A \to B$ in $\tilde S_n\mc{A}$, there exists a weak equivalence
    $h:X \to U(A)$ in $\tilde S_n\C$ such that $U(f)h = U(g)h$, then there exists a weak equivalence
    $\tilde{h}\colon Z\to A$ in $\tilde S_n \mc{A}$ such that $h$ factors through $U(\tilde{h})$ and such that $f\tilde{h}= g\tilde{h}$.

  \item Given any diagram
    \[\xymatrix{U(A) \ar@{^{(}->}[d]_{U(f)} & X \ar@{^{(}->}[d] \ar[l]_-\sim
        \ar[r]^-\sim & U(B) \ar@{^{(}->}[d]^{U(g)} \\
        U(A') & X' \ar[l]_-\sim \ar[r]^-\sim & U(B')}\]
    there exists a diagram
    \[\xymatrix{
        A \ar@{^{(}->}[d]_f & C \ar[l] \ar[r] \ar@{^{(}->}[d] & B \ar@{^{(}->}[d]^g \\
        A' & C' \ar[l] \ar[r] & B'
      }\]
    in $\mc {A}$ making the diagram
    \[\xymatrix{
        & & X \ar@{.>}[lld]|-\sim \ar@{.>}[ld]|-\sim \ar@{.>}[rd]|-\sim  \ar@{^{(}.>}[dd]\\
        U(A)  \ar@{^{(}->}[dd]_{U(f)} & U(C) \ar[l] \ar[rr]
        \ar@{^{(}->}[dd]& & U(B)  \ar@{^{(}->}[dd]^{U(g)} \\
        & & X' \ar@{.>}[lld]|-<<<<<<<\sim \ar@{.>}[ld]|-\sim \ar@{.>}[rd]|-\sim \\
        U(A') & U(C') \ar[l] \ar[rr] & & U(B')
      }\]
    commute.
  \end{enumerate}
  Then $K(U)$ is an equivalence.
\end{lemma}

\begin{proof}
  It suffices to prove that, for $n\ge 0$, the map
  $w\widetilde{S}_n(\mc{A})\to w\widetilde{S}_n(\C)$ induces a weak
  equivalence on geometric realizations. By Quillen's Theorem A \cite[Theorem
  A]{quillen}, it suffices to show that for any
  $\alpha\in w\widetilde{S}_n(\C)$, the undercategory
  $\alpha/w\widetilde{S}_n(\mc{A})$ is cofiltering.  Recall that $\alpha$ is
  represented by a diagram
  \[\xymatrix{X_1 \ar@{^{(}->}[r] &\cdots  \ar@{^{(}->}[r] & X_n}.\]

  For this, we first observe that by condition (1) the category
  $\alpha/w\widetilde{S}_n(\mc{A})$ is non-empty.  We use the second version of
  (1); the proof for the first works analogously.  Since $U$ is surjective on
  cofibrations by (1), there exists a cofibration $A_1 \to A_2$ that maps to
  $X_1 \to X_2$.  Now, since $U(A_2) = X_2$, there exists a cofibration
  $A_2 \rto A_3$ whose image under $U$ is $X_2 \rto X_3$.  Working inductively,
  we see that there is a sequence $A_1 \to\cdots\to A_n$ of cofibrations whose
  image under $U$ is $\alpha$.  To show that $\alpha/w\widetilde{S}_n(\mc{A})$
  is nonempty we take this sequence and the identity map from $\alpha$.

  We now show that $\alpha/w\tilde S_n(\mc{A})$ is co-filtering. We start by showing that any two parallel arrows are equalized by a third.  An object $A\in \alpha/w\tilde
  S_n(\mc{A})$ is a diagram
  \[\xymatrix{
      X_1 \ar@{^{(}->}[r] \ar[d]_\sim &\cdots  \ar@{^{(}->}[r] & X_n \ar[d]^\sim\\
      U(A_1) \ar@{^{(}->}[r] &\cdots  \ar@{^{(}->}[r] & U(A_n)}.\]
  Given two such objects $A$ and $B$, and two morphisms $f,g:A \rto B$, we must
  show that there exists a map $h\colon Z\to A$ in $\alpha/w\tilde S_n(\mc{A})$ such that $fh = gh$. This is exactly guaranteed by condition (2).

  It remains to show that for every pair of objects in
  $\alpha /w\widetilde{S}_n(\mc{A})$, there exists a third object which maps into
  each of them.  This is guaranteed by condition (3).
\end{proof}

Our main example is the $SW$-category of varieties together with
a choice of compactification; we show that forgetting the choices induces
an equivalence on $K$-theory.

\begin{defn} \label{def:Varc}
  Let $k$ be a field. We define the $SW$-category $\Varc_k$ as follows. Objects
  of $\Varc_k$ are open embeddings $X \xrightarrow{\circ} \overline{X}$
  where $X$ is a $k$-variety and $\overline{X}$ is a proper
  $k$-variety. Morphisms $(X \xrightarrow{\circ} \overline{X}) \rto (Y
  \xrightarrow{\circ} \overline{Y})$ are commuting squares
  \[
  \xymatrix{
    X \ar[r]^{\circ}\ar[d]_f & \overline{X} \ar[d]^{\overline f}\\
    Y \ar[r]^{\circ} & \overline{Y}
  }.
  \]
  A morphism
  $(X \xrightarrow{\circ} \overline{X}) \rto (Y \xrightarrow{\circ}
  \overline{Y})$ is a
  \begin{description}
  \item[cofibration] if  $f$ and  $\bar f$ are closed embeddings,
  \item[complement] if $f$ is an open embedding and $\overline{f}$ is a closed
    embedding, and
  \item[weak equivalence] if $f$ is an isomorphism.
  \end{description}
  A sequence
  $(Z \xrightarrow{\circ} \overline{Z}) \longrightarrow (X \xrightarrow{\circ}
  \overline{X}) \xleftarrow{\circ} (U \xrightarrow{\circ} \overline{U})$ is a
  subtraction sequence if the left map is a cofibration, the right map is a complement, $Z \hookrightarrow X \xleftarrow{\circ} U$ is a subtraction sequence in $\Var_k$, and the closed embedding $\overline{U}\to \overline{X}$ has set-theoretic image equal to the closure of $\overline{X}-\overline{Z}$ in $\overline{X}$.
\end{defn}

\begin{lemma}
  The category $\Varc_k$ with cofibrations, complements, weak equivalences, and
  subtraction defined as above satisfies the axioms of an $SW$-category.
\end{lemma}
\begin{proof} We verify the
  axioms from \cite{campbell} in turn. Note that limits and colimits are computed pointwise in $\Varc_k$. As a result, \cite[Defn. 3.7(1)]{campbell} is automatically satisfied. The axiom  \cite[Defn. 3.7(2)]{campbell} follows immediately because for any varieties $X,Y$, the embedding $X\to X\coprod Y$ is both open and closed. Axiom  \cite[Defn. 3.7(3)]{campbell} is also immediate, because isomorphisms in $\Varc_k$ are pointwise, and any isomorphism of varieties is simultaneously an open and closed embedding. Axiom  \cite[Defn. 3.7(4)]{campbell} holds for the same reason as in $\Var_k$. We now verify the remaining axioms in turn.
  \begin{itemize}
  \item \cite[Defn. 3.7(5a)]{campbell} Let $(Z\xrightarrow{\circ} \bar{Z})\into (X\xrightarrow{\circ} \bar X)$ be a cofibration.  Then $Z\into X$ determines the open embedding $X-Z\into X$ up to unique isomorphism, and similarly, $\bar Z\into \bar{X}$ determines the closed embedding $\overline{\bar X-\bar Z}\to \bar X$ up to unique isomorphism.
  \item \cite[Defn. 3.7(5b)]{campbell} Given a diagram
    \[
        \xymatrix{
            & (W,\bar{W}) \ar[d] \\
            (Z,\bar{Z})\, \rcofib[r] & (X,\bar{X}) & (U,\bar U) \ar[l]_\circ
        }
    \]
    where the bottom row is a subtraction sequence, we have a subtraction sequence
    \[
      \xymatrix{
             W\times_X Z\, \rcofib[r] & W & W \times_X Y \ar[l]_-\circ
      }
    \]
    in $\Var_k$. Moreover, because taking closure commutes with pullback, we have
    \begin{align*}
        \bar W\times_{\bar X} \bar U &\cong \bar W \times_{\bar X} \overline{\bar X - \bar Z}\\
        &=\overline{\bar W\times_{\bar X}(\bar X-\bar Z)}\intertext{where the outer $\overline{(-)}$ denotes closure in $\bar{W}$. Further, because the pullback of the complement of a subvariety is the complement of its pullback, this is isomorphic to}
        &=\overline{\bar W - \bar W\times_{\bar X} \bar Z}.
    \end{align*}
    Therefore, $\xymatrix{(W\times_X Z,\bar W\times_{\bar X}\bar Z) \rcofib[r]&
      (W,\bar W) & \ar[l]_\circ (W\times_X U,\bar W\times_{\bar X} \bar U)}$ is a subtraction sequence.
  \item \cite[Defn. 3.7(5c)]{campbell} Given a Cartesian square
      \[
            \xymatrix{
                (W,\bar W) \rcofib[r] \rcofib[d] & (X,\bar X) \rcofib[d] \\
                (Y,\bar{Y}) \rcofib[r] & (Z,\bar{Z})
                }
      \]
      in which all maps are cofibrations, the unique map $X-W\into  Z-Y$ is an open embedding. Further, there is a unique map $\bar X - \bar W\to \bar Z-\bar Y$ and, because $\bar X\to \bar Z$ is a closed emebdding, this extends uniquely to a closed embedding
      \[
        \overline{\bar X-\bar W}\to \overline{\bar Z-\bar Y}.
      \]
    \end{itemize}
    Axioms \cite[Defn. 3.13(1)]{campbell}, \cite[Defn. 3.13(2)]{campbell}, and \cite[Defn. 3.13(3)]{campbell} hold for the same
    reason they hold in $\Var_k$, namely \cite[Corollary 3.9]{schwede} and the
    remarks following.  Axioms \cite[Defn. 3.24(2)]{campbell} and \cite[Defn. 3.24(3)]{campbell} follow
    immediately from the definition of weak equivalences in $\Varc_k$ along with
    Axioms \cite[Defn. 3.13(1)]{campbell} and \cite[Defn. 3.7(5a)]{campbell} respectively.
\end{proof}

\begin{example} \label{ex:compactification}
  Let $\C = \Var_k$ and let $\mc{A} = \Varc_k$, with $U$ being the forgetful
  functor.  We claim that the conditions of Lemma~\ref{lem:Kequiv} hold in this example.  We check the
  conditions in turn:
  \begin{enumerate}
  \item By Nagata \cite[Theorem 4.3]{Nagata}\footnote{For a modern treatment of
      Nagata's Theorem, see \cite{Conrad}, esp. Theorem 4.1.} every $k$-variety
    $X$ admits an open embedding $X\xrightarrow{\circ} \overline{X}$ into a
    proper $k$-variety $\overline{X}$. Thus $U$ is surjective on objects.  Moreover, for any cofibration
    $X \hookrightarrow X'$, if $\bar X'$ is a compactification of $X'$ then the
    closure of $X$ in $\bar X'$ is a compactification of $X$.  Thus condition
    (a) holds.
  \item Weak equivalences in $\Var_k$ are isomorphisms. In this case this
    condition says that given two parallel maps $f$ and $g$ in $\tilde S_n\Varc_k$
    \[\xymatrix{
      (X_1,\bar X_1) \ar@{^{(}->}[r] \ar[d] &\cdots  \ar@{^{(}->}[r] & (X_n,\bar X_n) \ar[d]\\
      (Y_1,\bar Y_1) \ar@{^{(}->}[r] &\cdots  \ar@{^{(}->}[r] & (Y_n,\bar Y_n)}.
    \]
    such that $f$ and $g$ are equal when restricted to $X_i$, then there exists a weak equivalence $h$ in $\tilde S_n\Varc_k$
    \[\xymatrix{
      (X_1,\overline{\bar X_1}) \ar@{^{(}->}[r] \ar[d] &\cdots  \ar@{^{(}->}[r] & (X_n,\overline{\bar X_n}) \ar[d]\\
      (X_1,\bar X_1) \ar@{^{(}->}[r] &\cdots  \ar@{^{(}->}[r] & (X_n,\bar X_n)}.
    \]
    such that $fh=gh$. But for this, we can take $\overline{\bar X_i}$ to be the closure of $X_i$ in $\bar X_i$, and $h$ to be given by the canonical inclusions. Then, for each $i$, the $ith$ components of $fh$ and $gh$ agree on a dense set in $\overline{\bar X_i}$, and thus are equal. Therefore $fh=gh$.
  \item This property states that given a closed embedding of varieties
    $X \hookrightarrow X'$ together with induced embeddings of two choices of
    compactification $(X,\bar X) \hookrightarrow (X', \bar{X'})$ and
    $(X, \bar{\bar X}) \hookrightarrow (X', \bar{\bar X'})$, there exists a
    third choice of embeddings of compactifications that dominates both of
    these.  Note that if we can find a compactification of $X'$ that dominates
    both $\bar{X'}$ and $\bar{\bar{X'}}$ then we can take the closure of $X$ in
    it to produce the desired embedding.  Thus all that we must show is that for
    any two choices of compactification, there is a third that dominates them.
    For example, we can take the closure of $X$ inside $\bar{X}\times\bar{X}'$,
    as in \cite[\S IV-10.5]{SGA4.5}).
  \end{enumerate}
\end{example}

\section{$\ell$-adic Cohomology}
\label{sec:prelim-coh}
In this section we being by recalling the standard facts we will need about $\ell$-adic cohomology with its continuous Galois action. We then review the Godement resolution, and use this to establish the key technical Lemma~\ref{lem:all-tech}. We take \cite{SGA4.5} and \cite{FK} as standard references.

\subsection{Continuous Galois Representations}
Let $k$ be a field, let $k^s$ be a separable closure, and let $\ell\neq \characteristic(k)$ be a prime. The separable Galois group $\Gal(k^s/k)$ is a profinite group, and canonically carries the profinite topology
\begin{equation*}
    \Gal(k^s/k)=\varprojlim_{L/k~\text{fin., sep.}} \Gal(L/k)
\end{equation*}
where the finite groups $\Gal(L/k)$ are discrete, and the limit is in the
category of topological groups. Let $R$ be a ring. Recall that for a (discrete)
$R$-module $A$, a \emph{continuous representation} $\Gal(k^s/k)\to \Aut_R(A)$ is
one which factors through a finite subgroup
\[
\Gal(k^s/k)\to\Gal(L/k)\to \Aut_R(A)
\]
for some separable $L/k$. Denote by
$\Repc(\Gal(k^s/k);R)$ the category of finitely generated continuous
representations of $\Gal(k^s/k)$ over $R$. Recall the following
(cf. \cite[Arcata II.4.4]{SGA4.5}).
\begin{proposition}\label{prop:contGal}
    Let $k$ be a field, and $k^s$ a separable closure. Let $R$ be a ring. Denote by $\Sh(\Spec(k);R)$ the category of \'etale sheaves of (discrete) finitely generated $R$-modules on $\Spec(k)$. The functor
    \begin{align*}
        \Sh(\Spec(k);R)&\to \Repc(\Gal(k^s/k);R)\\
        F&\mapsto F|_{\Spec(k^s)}
    \end{align*}
    is an equivalence of categories.
\end{proposition}

We are especially interested in continuous $\ell$-adic representations. Recall the following reformulation of the category of finitely generated $\Z_\ell$-modules (we follow the presentation of \cite[Ch I.12]{FK}). We consider diagrams
\begin{equation*}
    F\colon \Z_\geq\to \Mod(\Z_\ell),
\end{equation*}
where $\Z_\geq$ is the category whose objects are integers and where there is
a unique morphism $m \rto n$ whenever $m \geq n$. The category of diagrams is an abelian category, and in particular has images, kernels, cokernels, etc. defined pointwise. For $r\in\Z$, denote by $F[r]$ the
shifted diagram, i.e. with $F[r]_m\coloneqq F_{r+m}$.
\begin{definition}
    A diagram $F\colon \Z_\geq\to \Mod(\Z_\ell)$ satisfies:
    \begin{enumerate}
        \item the \emph{Mittag--Leffler} (\emph{ML}) condition if for every $n$, there exists $t\ge n$ such that for all $m\ge t$,
            \begin{equation*}
                \Image(F_m\to F_n)=\Image(F_t\to F_n),
            \end{equation*}
        \item the \emph{Mittag--Leffler--Artin--Rees} (\emph{MLAR}) condition if there exists some $t\ge 0$ such that for all $r\ge t$,
            \begin{equation*}
                \Image(F[r]\to F)=\Image(F[t]\to F).
            \end{equation*}
    \end{enumerate}
\end{definition}

\begin{definition}\label{def:MLAR}
    Define $\Pro_{MLAR}(\Z_\ell)$ to be the category in which objects are MLAR diagrams in which each $F_n$ is torsion, and for for two such diagrams $F$ and $G$,
    \begin{equation*}
        \hom_{\Pro_{MLAR}(\Z_\ell)}(F,G)\coloneqq\varinjlim_{r\ge 0} \hom(F[r],G).
    \end{equation*}
\end{definition}

The key purpose of Mittag--Leffler diagrams is that on such diagrams, the inverse limit is an exact functor.   The ML property is frequently satisfied. For instance, if all the modules $F_n$ are finite length, then $F$ is an ML diagram.
\begin{definition}\label{def:ell}
    An \emph{$\ell$-adic} diagram is a diagram $F$ such that $F_n=0$ for $n<0$ and for all $n$
    \begin{enumerate}
        \item $F_n$ is a module of finite length,
        \item $\ell^{n+1} F_n=0$, and
        \item the map $F_{n+1}\to F_n$ induces an isomorphism
            \begin{equation*}
                F_{n+1}/\ell^{n+1} F_{n+1}\cong F_n.
            \end{equation*}
    \end{enumerate}
    An \emph{A--R $\ell$-adic} diagram is any object of $\Pro_{MLAR}(\Z_\ell)$ which is isomorphic to an $\ell$-adic diagram. Denote by $\AR(\ell)\subset\Pro_{MLAR}(\Z_\ell)$ the full sub-category of A--R-$\ell$-adic diagrams.
\end{definition}

We can now give the promised reformulation of the category $\Modf(\Z_\ell)$ of finitely generated $\Z_\ell$-modules (see e.g. \cite[Proposition I.12.4]{FK}).
\begin{proposition}\label{prop:AR}
    The inverse limit
    \begin{equation*}
        \varprojlim\colon \AR(\ell)\to \Modf(\Z_\ell)
    \end{equation*}
    is an equivalence of exact categories.
\end{proposition}

Recall that $\Z_\ell$ is a profinite ring, with the profinite (equivalently ``adic'') topology.  Similarly, for any finitely generated $\Z_\ell$-module $A$, the group $\Aut_{\Z_\ell}(A)$ is canonically a topological group, with the profinite topology. We can use the proposition to give a similar reformulation of the category $\Repc(\Gal(k^s/k);\Z_\ell)$ of continuous representations $\Gal(k^s/k)\to\Aut_{\Z_\ell}(A)$. Mutatis mutandis, we obtain from Definitions \ref{def:MLAR} and \ref{def:ell} a notion of $\ell$-adic diagrams of continuous $\Gal(k^s/k)$ representations, and a category $\RepcAR(\Gal(k^s/k);\ell)$ of such. Concretely, objects are given by diagrams
\begin{equation*}
    \cdots\to F_n\to F_{n-1}\to\cdots
\end{equation*}
where for each $n$, $F_n$ is a continuous representation of $\Gal(k^s/k)$ in finitely generated (discrete) $\Z/\ell^n\Z$-modules and the analogous conditions to those of Definition \ref{def:ell} hold. Analogously to Proposition \ref{prop:AR}, we have the following.
\begin{proposition}\label{prop:GalAR}
    The inverse limit
    \begin{equation*}
        \varprojlim\colon \RepcAR(\Gal(k^s/k);\ell)\to \Repc(\Gal(k^s/k);\Z_\ell)
    \end{equation*}
    is an equivalence of exact categories.
\end{proposition}

\subsection{$\ell$-adic Sheaves}

We now extend the above to sheaves. We consider schemes $X$ for which $\ell\in \Oc(X)$ is invertible. Mutatis mutandis, we obtain from Definition \ref{def:MLAR} a definition of MLAR diagrams of $\ell$-torsion \'etale sheaves, and the category $\Pro_{MLAR}(\Sh(X),\Z_\ell)$ of such.
\begin{definition}
    Let $X$ be a scheme and $\ell$ a prime invertible on $X$. An \emph{$\ell$-adic sheaf} on $X$ is diagram $F\colon\Z_\geq\to \Sh(X;\Z)$ of \'etale sheaves of abelian groups on $X$ such that
    \begin{enumerate}
        \item the sheaves $F_n$ are constructible for all $n$,
        \item $F_n=0$ for $n<0$, and
        \item the map $F_{n+1}\to F_n$ induces isomorphisms
            \begin{equation*}
                F_{n+1}\otimes_{\Z}\Z/\ell^n\Z\cong F_n.
            \end{equation*}
    \end{enumerate}
    An \emph{A--R-$\ell$-adic sheaf} is any object of $\Pro_{MLAR}(\Sh(X),\Z_\ell)$ which is isomorphic to an $\ell$-adic sheaf. We denote the category of A--R-$\ell$-adic sheaves by $\Sh(X;\Z_\ell)$.
\end{definition}

The following is the key theorem we use for $\ell$-adic sheaves (cf. \cite[Theorem I.12.15]{FK}). Recall that a map $f\colon X\to S$ is {\em compactifiable} if $f$ can be factored as an open embedding $j\colon X\into\bar{X}$ followed by a proper map $\bar{f}\bar{X}\to S$. Given a compactifiable map $f\colon X\to S$ and $\nu\ge 0$, recall (cf. \cite[Definition/Proposition I.8.6]{FK}) that $R^\nu f_!$ denotes the functor
\[
    R^{\nu}\bar{f}_\ast \circ j_!\colon \Sh(X)\to \Sh(S)
\]
By {\em loc. cit.}, this functor is independent of the choice of $j$ and $\bar{f}$, and is referred to as the $\nu$th higher direct image with compact supports.

\begin{theorem}[Finiteness Theorem for $\ell$-adic Sheaves]
    Let $f\colon X\to S$ be a compactifiable map, let $\ell$ a prime invertible on $X$, and $n\mapsto F_n$ an A--R-$\ell$-adic sheaf on $X$. For $\nu\ge 0$, let Then for all $\nu\ge 0$, the system $n\mapsto R^\nu f_!(F_n)$ is an A--R-$\ell$-adic sheaf on $S$.
\end{theorem}

Using the smooth base change theorem (see e.g. \cite[Theorem I.7.3]{FK}) and Propositions \ref{prop:contGal} above, we immediately deduce the following.
\begin{corollary}\label{cor:ARres}
    Let $k$ be a field, $k^s$ a separable closure, and let $X$ be a variety of finite type over $k$. Then for any A--R-$\ell$-adic sheaf $F$ on $X$, a choice of compactification $j\colon X\into \bar{X}$ and a choice of functorial flasque resolution $Q_{\bar X}$ on $\Sh(\bar X;\Z_\ell)$ determine an extension of the assignment
    \begin{equation*}
        F\mapsto H^\ast_c(X_{/k^s};F)
    \end{equation*}
    to a functor
    \begin{align*}
        \Sh(X;\Z_\ell)&\to\Chb(\RepcAR(\Gal(k^s/k);\ell))\\
        F&\mapsto \Gamma_{\bar X} Q_{\bar X}(j_!F)_{|_{k^s}}
    \end{align*}
\end{corollary}

\subsection{The Godement Resolution}\label{sec:gode}
Recall that the Godement resolution $\mc{G}^\bullet(\mc{F})$ of an \'etale sheaf $\mc{F}$ on a $k$-variety $X$ is defined inductively as follows (see e.g. \cite[p. 129]{FK}). Fix algebraic closures $\Omega_0:=\bar{k}$ and $\Omega_n:=\bar{k(t_1,\ldots,t_n)}$ for all $n$. Let $M_X$ be the set of geometric points
\[
    M_X:=\{x\colon \Spec(\Omega_n)\to X ~|~n\in\N\}
\]
Note that every point in $X$ is associated to at least one geometric point in $M_X$ (since algebraically closed fields over $k$ are distinguished up to isomorphism only by their transcendance degree, e.g. \cite[Theorem VI.1.12]{hungerford}).  In particular, this is what allows us to define the Godement resolution with respect to this set of points of the \'etale site of $X$.

Now define
\[\mc{G}_X^0(\mc{F}):=\prod_{x\in M_X} x_\ast x^\ast \mc{F}\]
Equivalently, $\mc{G}_X^0(\mc{F})=\delta_{X\ast} \delta_X^\ast \mc{F}$ where $\delta_X\colon \coprod_{M_X}x \to X$ is the inclusion of points.

The functor $\mc{G}_X^0(-)$ is exact (it is exact on stalks), and comes with a diagonal inclusion
\[
    \mc{F}\to \mc{G}_X^0(\mc{F})
\]
given by mapping a section to its stalks. Define
\[
    \mc{G}_X^1(\mc{F}):=\mc{G}_X^0(\coker(\mc{F}\to\mc{G}_X^0(\mc{F}))).
\]
In general, define
\[
    \mc{G}_X^{i+1}(\mc{F}):=\mc{G}_X^0(\coker(\mc{G}_X^{i-1}(\mc{F})\to \mc{G}_X^i(\mc{F})))
\]
Then the assignment $\mc{F}\mapsto \mc{G}_X^\bullet(\mc{F})$ has the following properties:
\begin{enumerate}
    \item it defines an exact functor $\mc{G}_X^\bullet\colon \Sh(X)\to\Chb(\Sh(X))$,
    \item the sheaves $\mc{G}^i_X(\mc{F})$ are flabby for all $i$,
    \item the map $\mc{F}\to\mc{G}_X^\bullet(\mc{F})$ is a resolution.
\end{enumerate}
In particular, for any ring $R$ the functor
\[
    S_X:=\Gamma_X \circ \mc{G}_X^\bullet\colon \Sh(X;R)\to\Chb(\Sh(\Spec(k);R))
\]
is exact. Using the smooth base change theorem (see e.g. \cite[Theorem I.7.3]{FK}) and Propositions \ref{prop:contGal} and \ref{prop:GalAR} above, we immediately deduce the following strengthening of Corollary~\ref{cor:ARres}.
\begin{corollary}\label{cor:Gode}
    Let $k$ be a field, $k^s$ a separable closure. Fix a Godement resolution $\mc{G}^\bullet_{(-)}$ for $k$-varieties as above, and write $S_{(-)}:=\Gamma_{(-)}\circ\mc{G}^\bullet_{(-)}$ as above. Let $X$ be a variety of finite type over $k$. Then for any A--R-$\ell$-adic sheaf $F$ on $X$, a choice of compactification $j\colon X\into \bar{X}$ determines an extension of the assignment
    \begin{equation*}
        F\mapsto H^\ast_c(X_{/k^s};F)
    \end{equation*}
    to an exact functor
    \begin{align*}
        \Sh(X;\Z_\ell)&\to\Chb(\RepcAR(\Gal(k^s/k);\ell))\\
        F&\mapsto S_{\bar X}(j_!F)_{|_{k^s}}
    \end{align*}
\end{corollary}

\subsection{A Technical Lemma}
We now establish the key technical lemma of the paper. As above, we fix a separable closure $k^s$ of $k$ and a Godement resolution $\mc{G}^\bullet_{(-)}$ for $k$-varieties. As above, we adopt the notation:
\[
    S_{(-)}:=\Gamma_{(-)}\mc{G}^\bullet_{(-)}
\]
i.e. $S_{(-)}$ is the strict model of the sheaf cohomology functor associated to the Godement resolution.

\begin{lemma} \label{lem:all-tech} Let $X$ be a $k$-variety, let $\iota\colon U\into X$ be the inclusion of an open subvariety. Let $j\colon X\rto Y$ be a proper map for which $j\circ\iota$ is an embedding. Then for all sheaves $\mc{F}$ on $U$, there is an induced isomorphism (of chain complexes of sheaves on $\Spec(k)$)
\[ j_\Gamma\colon S_X (\iota_! \mc{F})\to^\cong S_Y (j_! \iota_!\mc{F}).\]
This isomorphism is natural in the following two senses:
\begin{enumerate}
    \item For any map of sheaves $g: \mc{F} \rto \mc{F}'$ on $U$, the diagram
        \[\xymatrix{
            S_X(\iota_!\mc{F}) \ar[r]^{j_\Gamma} \ar[d]_{S_X \iota_!(g)} & S_Y (j_! \iota_!\mc{F})
            \ar[d]^{S_Y j_! \iota_! (g)} \\
            S_X (\iota_!\mc{F}') \ar[r]^{j_\Gamma} & S_Y (j_! \iota_! \mc{F}')
        }\]
        commutes.
    \item Given a commuting diagram
        \begin{equation*}
            \xymatrix{
                U \ar[r]^{(j_1)|_U} \ar[d]^{\iota_1} & V \ar[d]^{\iota_2} \\
                X \ar[r]^{j_1} & Y \ar[r]^{j_2} & Z
            }
        \end{equation*}
        where $\iota_i$ is an open embedding and $j_i$ is a proper map such that $j_i\circ \iota_i$ is an embedding for $i=1,2$, and $j_1(U)\subset V$, then for any sheaf $\mc{F}$ on $U$,
        \begin{equation*}
            (j_2)_\Gamma\circ(j_1)_\Gamma=(j_2\circ j_1)_\Gamma.
          \end{equation*}
\end{enumerate}
\end{lemma}

\begin{proof}
  Since $j$ is proper, $j_! = j_*$, so it suffices to prove the statement of the
  lemma with $j_*$ substituted for $j_!$.

  Now let $\iota\colon U\to X$ be an open embedding, let $\mc{F}$ be a sheaf on $U$, and let $j\colon X\to Y$ be a proper map such that $j\iota\colon U\to Y$ is an embedding. We claim that there is a canonical isomorphism (of chain complexes in $\Sh(Y)$)
  \[
        j_\ast\mc{G}_X^\bullet(\iota_!\mc{F})\to^\cong \mc{G}_Y^\bullet(j_\ast \iota_!\mc{F}).
  \]
  Granting the claim, we obtain $j_\Gamma$ by applying $\Gamma_Y$ to this isomorphism and pre-composing with the natural isomorphism $\Gamma_X(-)\cong \Gamma_Y\circ j_\ast(-)$.

  In the notation of Section~\ref{sec:gode}, the key observation underpinning the claim is that any map of $k$-varieties $f\colon W\to Z$ determines a map of sets $f\colon M_W\to M_Z$ (just by post-composing $f$ with each map $x\colon\Spec(\Omega_n)\to W$, and thus a commuting diagram (of non-finite type varieties)
    \[
        \xymatrix{
        \coprod_{M_W} w \ar[r]^{\delta_f} \ar[d]^{\delta_W} & \coprod_{M_Z} z \ar[d]^{\delta_Z} \\
        W \ar[r]^f & Z
        }
    \]
    If $f$ is an open embedding, then by the definition of $f_!$, we have a
    canonical ``base change'' isomorphism
    \begin{equation*}
        \delta_Z^\ast f_! \mc{F}\cong \delta_{f\ast}\delta_W^\ast \mc{F}.
    \end{equation*}
    Since $f$ is an embedding, $\delta_f$ is an embedding as well, and is
    both open and closed because the varieties in question are discrete. Therefore
    $\delta_{f!}=\delta_{f\ast}$.

    We are now ready to prove the claim by induction.  For the base of the induction,
    \begin{align*}
        j_\ast \mc{G}^0_X(\iota_!\mc{F})&:=j_\ast \delta_{X\ast}\delta_X^\ast \iota_!\mc{F} \\
        &\cong j_\ast \delta_{X\ast}\delta_{\iota\ast} \delta_U^\ast \mc{F} \qquad \text{(because $\iota$ is an embedding)}\\
        &\cong \delta_{Y\ast} \delta_{j\ast} \delta_{\iota\ast} \delta_U^\ast \mc{F} \qquad \text{(by the functoriality of $(-)_\ast$)}
    \end{align*}
    Finally, because $j\circ \iota$ is an embedding, by inspection of the definitions, we see that there is a canonical isomorphism
    \begin{equation*}
        \delta_{j\ast}\delta_{\iota\ast}\delta_U^\ast\mc{F}\cong \delta_Y^\ast j_\ast \iota_!\mc{F}
    \end{equation*}
    (i.e. $y_\ast y^\ast j_\ast \iota_!\mc{F}=y_\ast y^\ast \mc{F}$ for $y\in M_U\subset M_Y$ and $y_\ast y^\ast j_\ast \iota_!\mc{F}=0$ otherwise.)
    We conclude that
    \begin{align*}
        j_\ast \mc{G}^0_X(\iota_!\mc{F})\cong \delta_{Y\ast}\delta_Y^\ast j_\ast \iota_!\mc{F}=:\mc{G}^0_Y(j_\ast\iota_!\mc{F}).
    \end{align*}
    This settles the base of the induction. For the induction step, note that in the argument above, we have
    \begin{align*}
        j_\ast \mc{G}^0_X(\iota_!\mc{F})&\cong j_\ast \delta_{X\ast}\delta_{\iota\ast} \delta_U^\ast \mc{F}
       \cong j_\ast \iota_!\delta_{U\ast}\delta_U^\ast \mc{F}
    \end{align*}
    where the second isomorphism follows by inspection from the definition  of $\iota_!$. This implies that
    \begin{align*}
        \mc{G}_X^1(\iota_!\mc{F})&:=\coker(\iota_!\mc{F}\to\iota_!\delta_{U\ast}\delta_U^\ast\mc{F})\cong \iota_!\mc{G}_U^1(\mc{F}).
    \end{align*}
    In particular, $\mc{G}_X^1(\iota_!\mc{F})$ is again an extension by $0$ of a
    sheaf on $U$. Because $j$ is proper, $j_\ast$ preserves colimits, so
    \begin{align*}
        j_\ast\mc{G}_X^1(\iota_!\mc{F})\cong \mc{G}_Y^1(j_\ast\iota_!\mc{F})
    \end{align*}
    and we can now apply the same argument as above to conclude that there exists a canonical isomorphism
    \begin{align*}
        j_\ast \mc{G}_X^i(\iota_!\mc{F})\cong \mc{G}_Y^i(j_\ast\iota_!\mc{F})
    \end{align*}
    for all $i$, thus proving the claim.

    It remains to show the two naturality properties. The naturality with respect to morphisms of sheaves $g\colon \mc{F}\to \mc{F}'$ follows immediately from the functoriality of the constructions above.

    For the second naturality property, a direct inspection of the construction above shows that the naturality property follows from the properness of the maps $j_i$ and the universal properties of cokernels and products (using that $\delta_{X\ast}\delta_X^\ast\mc{F}=\prod_{x\in M_X} \mc{F}_x$). Concretely, the naturality follows by writing out the explicit definitions of the maps and sheaves on $\Spec(k)$ and then repeatedly using that if $I\to^{\iota} J\to^j K$ are maps of sets such that $\iota \colon I\to J$ and $j\iota\colon I\to K$ are injective, then for any sheaf of abelian groups $A$ on the discrete space $I$, there are canonical isomorphisms
    \begin{align*}
        \prod_I A_i&\cong \prod_I A_i\times \prod_{J-I} 0\cong \prod_J (\iota_! A)_j\\
        &\cong \prod_I A_i\times\prod_{K-I} 0 \cong \prod_K ((j\iota)_! A)_k.
    \end{align*}
\end{proof}

\section{The General Approach, by example} \label{sec:approach}

In this section we prove a warm-up to our main result in order to illustrate the general approach. Recall the SW-category $\Varc_k$ of Definition~\ref{def:Varc} (consisting of $k$-varieties with a choice of compactification), and the exact functor
\[
    \Varc_k\to\Var_k
\]
given by forgetting the compactification.

We now prove the following ``warm-up'' to Theorem~\ref{thm:main}, in order to illustrate our general approach.
\begin{theorem} \label{thm:derEuler}
    Let $k$ be a subfield of $\CC$, and $R$ a commutative ring.  Let $\Chb(R)$ be the category of homologically finite and bounded chain complexes of $R$-modules. Let $[C_c^\bullet(-;R)]$ denote the class of compactly supported $R$-valued singular cochains in the homotopy category. Then the functor
    \begin{align*}
        \Var_k^{\times}&\to \Ho(\Chb(R))^{\op}\\
        X&\mapsto [C_c^\bullet(X(\CC);R)]
    \end{align*}
    admits a strict model as a span of weakly $W$-exact functors
    \[
        \xymatrix{
            \Var_k & \Varc_k \ar[l]_\sim \ar[r]^-G & \Chb(R)^{\op}
            }
    \]
    where the left-facing map induces an equivalence on $K$-theory.
\end{theorem}
\begin{proof}
    For conciseness we abuse notation and write $X$ for $X(\CC)$ throughout this proof. We also suppress the coefficients on cohomology, which are always be taken to be $R$.

    The functor $U\colon\Varc_k\to\Var_k$ forgets the choice of compactification.

    We define the functor $G$ as follows.  Let \[G(X,\bar X) = C_{\mathrm{sing}}^*(\bar X, \bar X - X),\] be the chain complex of singular cochains with coefficients in $R$ on the topological space
    $\bar X$ which vanish on all chains contained in $\bar{X}-X$.  The functor
    \[
        G^!:\co(\Varc_k) \rto \Chb(R)^\op
    \]
    sends a closed embedding $(Z, \bar Z) \rto (X, \bar X)$ to the usual pullback on cohomology. Note that
    this is well-defined, since (as $Z$ is closed in $X$)  $\bar Z - Z \subseteq \bar X - X$. The functor
    \[
        G_!:\comp(\Varc_k) \rto \Chb(R)^\op
    \]
    sends an open embedding $(U, \bar U) \rto (X, \bar X)$ to the map which extends a cochain on $\bar U$
    to a cochain on $\bar X$ by defining it to be $0$ on all cochains which are not contained in $\bar U$.  The functor $G^w: \w(\Varc_k) \rto \Chb(R)^\op$ is defined similarly, with the analogous observation that a weak equivalence $(X,\bar X) \to (X, \bar{\bar X})$ produces a map of pairs $(\bar X, \bar X - X) \rto (\bar{\bar X}, \bar{\bar X}-X)$.

    By construction, the span
    \[
        \xymatrix{
            \Var_k & \Varc_k \ar[l] \ar[r]^-G & \Chb(R)^{\op}
            }
    \]
    provides a strict model for the functor
    \[
        X\mapsto [C_c^\bullet(X(\CC);R)]
    \]
    It remains to prove that $G$ is actually weakly $W$-exact.  Conditions (1)-(4) hold by definition.  We check the others in turn.
    \begin{itemize}
        \item[(5)] By the definitions of the maps, it suffices to check that the diagram commutes on the compactification components.  Thus the statement we need to check is that for any cartesian diagram
            \[\xymatrix{\bar X \ar[r]^j \ar[d]_i & \bar Z \ar[d]^{i'} \\
                \bar Y \ar[r]^{j'} & \bar W}\]
            where all maps are closed embeddings, $(i')^!\circ j'_! = j_! \circ i^!$. In this case, both of the compositions around the diagram take a cochain $\alpha: C_*(\bar Y, \bar Y - Y) \rto R$ to the cochain $\alpha':C_*(\bar Z,\bar Z - Z) \rto R$ which takes a chain $\sigma: \Delta^n \to \bar Z$ to $0$ if $\sigma$ doesn't factor through $\bar X$, and to $\alpha(i\circ \sigma)$ if it does.  Thus the diagram commutes.
  \item[(6)] Suppose that we are given a subtraction sequence \[\xymatrix{(Z,\bar Z) \rcofib[r]^i & (X, \bar X) & (U, \bar U) \ar[l]^-{\circ}_-j}.\]
        Note that we have the following commutative diagram of pairs of spaces:
        \[\xymatrix{ (\bar U, \bar U - U) \ar[d] \\
            (\bar X, \bar X - U) & (\bar X, \bar X - X) \ar[l] & (\bar X - U, \bar X-X) \ar[l] \\
            & & (\bar Z, \bar Z - Z) \ar[u] \ar[lu]}\]
        Applying $C^*$ gives an exact sequence of cochain complexes across the middle, as it is the sequence associated to the triple $(\bar X, \bar X - U, \bar X - X)$.  By excision, all vertical isomorphisms
        become quasi-isomorphisms on cochains.  Thus we have the diagram
        \[\xymatrix{ C^*(\bar U, \bar U - U) \ar[rd]^{i_!} \ar@/_/[d]_\sim  \\
            C^*(\bar X, \bar X - U) \ar[r] \ar@{.>}[u] & C^*(\bar X, \bar X
            - X) \ar[r] \ar[rd]_{j^!} & C^*(\bar X - U, \bar X-X) \ar[d]^\sim\\
            & & C^*(\bar Z, \bar Z - Z) }\]
        where the dotted arrow is the result of applying $C^*$, and the curved arrow is the quasi-inverse that extends chains by $0$.  Both of the solid arrow triangles commute.  Thus the sequence
        \[G(U, \bar U) \to^{i_!} G(X,\bar X) \to^{j^!} G(Z,\bar Z)\] is weakly equivalent to an exact sequence, and thus is weakly exact.
  \item[(7)] The condition for complement maps holds because $G^w$ and $G^!$ are given by the same functor.  The condition for cofibrations holds because in both cases we extend a cochain on $\bar{X'}$ to a cochain on $\bar Y$ which is zero outside of $X'$; the only difference is that applying the maps in
    one direction extends the cochain to $\bar{Y'}$ first, and then to $\bar{Y}$, and in the other direction it extends to $\bar X$ first and then to $\bar Y$.
  \end{itemize}
\end{proof}

Now let $R$ be any ring, and let $\chi_R$ denote the compactly supported $R$-valued Euler characteristic of homologically finite topological spaces, i.e.
\[
    \chi_R(X):=\sum_{i=0}^\infty(-1)^i\rk_R H_c^i(X;R).
\]
We now deduce Theorem~\ref{cor:HWeuler} as a corollary of Theorem~\ref{thm:derEuler}.
\begin{corollary} \label{cor:Reuler}
    Let $k$ be a subfield of $\Cb$. Then the strict model of Theorem~\ref{thm:derEuler} determines a contractible space of maps of $K$-theory spectra
    \begin{equation*}
	    \mathbf{X}_R: K(\Var_k)\to K(R),
    \end{equation*}
    which fit into a commuting square
    \begin{equation*}
          \xymatrix@C3em{
            K_0(\Var_{\Cb}) \ar[r]^-{\pi_0 \mathbf{X}_R} \ar[d]^{\chi_R} &
            K_0(R) \ar[d]^{\rk}\\
            \Z \ar@{=}[r] & \Z}
	\end{equation*}
    after applying $\pi_0$.
\end{corollary}
\begin{proof}
    As discussed in Lemma~\ref{lem:chains}, the inclusion of an exact category as the chain complexes concentrated at $0$ is an equivalence, with the ``Euler characteristic'' as the inverse.  Thus, we obtain from Theorem~\ref{thm:derEuler} a diagram of $W$-exact functors
    \[
        \xymatrix{
            \Var_k & \Varc_k \ar[l] \ar[r]^G & \Chb(R)^{\op} & \Modf(R)^{\op} \ar[l]
        }
    \]
    By Example~\ref{ex:compactification} and Lemma~\ref{lem:chains}, the left-facing arrows induce weak equivalences on $K$-theory.

    Since a weak equivalence between fibrant-cofibrant spectra has a contractible space of inverses, we obtain a contractible space of maps
    \[
        \mathbf{X}_R:\colon K(\Var_k)\to^h K(\Varc_k)\to^{K(G)} K(\Chb(R))\to^{\chi} K(\Modf(R))=K(R)
    \]
    where $h$ is any choice of inverse equivalence to $K(\Varc_k)\to K(\Var_k)$ and $\chi$ is any choice of inverse to $K(\Modf(R))\to K(\Chb(R))$. The commuting square after applying $\pi_0$ now follows by inspection.
\end{proof}

Granting the statement of Theorem~\ref{thm:main}, we can similarly deduce the following generalization of Corollary~\ref{cor:HWzeta}.
\begin{corollary} \label{cor:gZeta}
    Let $k$ be a field, $k^s$ a separable closure, $\ell\neq \characteristic(k)$ a prime, and $g\in\Gal(k^s/k)$ any element.  Then the strict model of Theorem~\ref{thm:main} determines a contractible space of maps of $K$-theory spectra
    \begin{equation*}
       g_\ast\circ \Zeta\colon K(\Var_{k})\to K(\Aut(\Z_\ell)).
    \end{equation*}
    Passing to $K_0$ and taking the characteristic polynomial, we recover the ``$g$-Zeta-function''
    \begin{align*}
        Z_g\colon K_0(\Var_k)&\to W(\Z_\ell)\\
        [X]&\mapsto \prod_{i=0}^{2\dim(X)}\det(1-g^\ast t;H^i_{\et,c}(X\times_{k} k^s;\Z_\ell))^{(-1)^i}.
    \end{align*}
\end{corollary}
\begin{proof}
    The map $g^\ast$ is the map induced on $K$-theory by the functor
    \begin{align*}
        g^\ast \colon\Repc(\Gal(k^s/k);\Z_\ell)&\to \Aut(\Z_\ell)\\
            \Gal(k^s/k)\circlearrowleft V&\mapsto (V,g)
    \end{align*}
    For the map $\Zeta$, consider the diagram of $W$-exact functors
    \[
        \xymatrix{
            \Var_k & \Varc_k \ar[l] \ar[r] & \Chb(\Repc(\Gal(k^s/k);\Z_\ell))^{\op} & \Repc^{fg}(\Gal(k^s/k))^{\op} \ar[l]
        }
    \]
    from Theorem~\ref{thm:main} (where the right-most arrow is again just inclusion of complexes concentrated in degree 0). The construction of
    \[
        \Zeta\colon K(\Var_k)\to K(\Repc(\Gal(k^s/k);\Z_\ell))
    \]
    now follows along identical lines to the construction of $X$ in the proof of Corollary~\ref{cor:Reuler}.

    The commuting square on $\pi_0$ now follows from results of Almkvist \cite{almkvist_3}, explicated beautifully in Grayson \cite{Gr79}.
\end{proof}

\begin{remark}
    Taking $k=\F_q$ and $g$ to be Frobenius, the ``$g$-Zeta-function'' is precisely the classical zeta function (see e.g. the discussion on p. 171-174 of \cite{FK}); in this special case Corollary~\ref{cor:gZeta} is exactly
    Corollary~\ref{cor:HWzeta}.
\end{remark}

\section{Proof of Theorem~\ref{thm:main}} \label{sec:longproof}
We now prove Theorem~\ref{thm:main}, along the lines indicated in the proof of Theorem~\ref{thm:derEuler}; the main difference is in the intricacy of the construction of the $W$-exact functor.

By Proposition~\ref{prop:GalAR}, it suffices to construct a $W$-exact functor
\begin{equation}\label{Wexactfun}
  F:\Varc_{k}\to \Chb(\RepcAR(\Gal(k^s/k);\ell))^\op
\end{equation}
such that the cohomology of the inverse limit of the A--R $\ell$-adic chain complex produces compactly supported $\ell$-adic cohomology (with the Galois action induced by the action on $k^s$).

\begin{remark}\mbox{}
    \begin{enumerate}
        \item We prove in this section that the functor which assigns compactly supported cochains, defined via the Godement resolution, to the constant $\ell$-adic sheaf $\Z_\ell$ takes subtraction sequences of $k$-varieties to exact sequences of homologically bounded chain complexes of sheaves on $\Spec(k)$.
        \item Our proof applies to the constant sheaf $\Z_\ell$, considered as a uniform system of sheaves on all $k$-varieties. The key property we use is that for any map $f\colon Y\to X$, we have an isomorphism $f^\ast \Z_{\ell,X}\cong \Z_{\ell,Y}$. Our proof does not apply to any collection of sheaves on $k$-varieties for which this identity ever fails. In particular, we do \textbf{not} prove that there is a chain level realization of any cohomology functor which is exact on the category of \emph{all} sheaves.
        \item Our proof exploits a key difference between sheaf cohomology and singular cohomology of topological spaces.  Namely, given a disjoint decomposition of a space $X=Z\cup (X-Z)$, it is not the case that a singular chain is either contained in $Z$ or in $X-Z$.  By contrast, every point of $X$ is contained in one piece of the decomposition or the other.  The Godement resolution is defined purely in terms of the points of the space; this is what allows our construction of compactly supported cochains to give an exact functor, as opposed to merely a weakly exact functor.
    \end{enumerate}
\end{remark}

As in Section~\ref{sec:gode}, we fix once and for all a Godement resolution as our functorial flasque resolution for \'etale sheaves over a variety $X$,
\begin{equation*}
    \mathcal{F}\longmapsto  \mathcal{G}_X^\bullet\mathcal{F}.
\end{equation*}

\begin{notation}
  In this section, we make the following shorthand definitions.
  \begin{description}
  \item[$A^n_X$] denotes, for a $k$-variety $X$, the sheaf $(\Z/\ell^n\Z)_{X|_{k^s}}$.
  \item[$S_X$] denotes the functor $\Gamma_X \mc{G}^\bullet_X$.  Note that this functor lands
    in the category of chain complexes, \emph{not} in the derived category, which is why we avoid the standard notation $R\Gamma$.
  \end{description}
\end{notation}

Following the discussion in Section \ref{sec:prelim}, we proceed by defining functors $F_n^w, F_!^n$ and $F^!_n$ taking values in $\Chb(\Repc(\Gal(k^s/k);\Zb/\ell^n\Zb))$. We then show these fit together to define an A--R-$\ell$-adic system of continuous Galois representations.  As all of the $\Gal(k^s/k)$-actions follow from the action on $k^s$, we omit these from the notation, but the implication should be that all functors record this data
as well.

Throughout the construction below, we make repeated use of the isomorphism (of chain complexes of sheaves on $\Spec(k)$) of Lemma~\ref{lem:all-tech}
\[
    j_\Gamma\colon S_X(\iota_!\mathcal{F})\to^\cong S_Y(j_!\iota_!\mathcal{F})
\]
associated to an open embedding $\iota U\into X$ and a proper map $j\colon X\to Y$. Recall that $(-)_!$ applied to an open embedding denotes the ``extension by 0'' functor.

We begin by constructing the functor
\begin{equation*}
    F_n^!\colon\comp(\Varc_{k})^\op\to \Chb(\Repc(\Gal(k^s/k);\Zb/\ell^n\Zb))^\op.
\end{equation*}
On objects, it is given by
\begin{equation}\label{Wexactobn}
  F_n^!(\gamma_X\colon X\xrightarrow{\circ}\bar{X})\coloneqq S_{\bar X}(\gamma_{X!}A^n_X).
\end{equation}

By definition, a complement map $j$ in $\Varc_k$ consists of a commuting square
\begin{equation*}
  \xymatrix{U \ar[r]^j \ar[d]_{\gamma_U} & X \ar[d]^{\gamma_X} \\ \overline{U}
    \ar@{->}[r]^{\bar j} & \overline{X}}
\end{equation*}
where $j$ is an open embedding and $\bar j$ is a closed embedding. Given such, we
obtain a map
\begin{equation*}
  F_n^!j: S_{\bar U}(\gamma_{U!}A^n_U)\to S_{\bar X}(\gamma_{X!}A^n_X)
\end{equation*}
via the following composition:
\[\xymatrix@C=5em@R=1em{
  S_{\bar U}( \gamma_{U!}A^n_U) \ar[r]^{\bar j_\Gamma} &
  S_{\bar X}( \bar j_! \gamma_{U!} A^n_U)
   \ar[r]^-\cong & S_{\bar X}( \gamma_{X!}j_!
   A^n_U)
   \\
   \ar[r]^-\cong &
   S_{\bar X}(\gamma_{X!}j_!j^\ast A^n_X)\ar[r]^-{S_{\bar X}\gamma_{X!}(\epsilon)}
   & S_{\bar X}(\gamma_{X!}A^n_X).
}\]
The first map is the isomorphism of Lemma~\ref{lem:all-tech}. The middle two isomorphisms come from the
canonical identification $\bar j_!\gamma_{U!}=\gamma_{X!}j_!$ and $A^n_U\cong j^*A^n_X$. The last morphism $\epsilon$ is the counit of the adjunction $j_!\dashv j^\ast$ (which exists because $j$ is an open embedding).

\begin{lemma}
    The assignment $j\mapsto F_n^!j$ is functorial on $\comp(\Varc_{k})^\op$.
\end{lemma}
\begin{proof}
  The definition immediately implies that identities are mapped to identities,
  so we only need to check that composition is respected. For this, note that a
  composable pair of cofibrations in $\Varc_{k}$ consists of a commuting diagram
  \begin{equation*}
    \xymatrix{U \ar[r]^{j_1} \ar[d]_{\gamma_U} & X \ar[r]^{j_2}
      \ar[d]^{\gamma_X} & Y \ar[d]^{\gamma_Y} \\
      \overline{U} \ar[r]^{\bar j_1} & \overline{X} \ar[r]^{\bar j_2} & \overline{Y}}
  \end{equation*}
  in which the maps $j_i$ are open embeddings and the maps $\bar{j_i}$ are closed embeddings. The above diagram yields the following commutative diagram:
  \[\xymatrix@C=3em{
      S_{\bar U}(\gamma_{U!} A^n_U) \ar[d]^-{\bar j_{1\Gamma}}
      \ar[rd]^{(\bar j_2\bar j_1)_\Gamma} \\
      S_{\bar X}(\bar j_{1!}\gamma_{U!} A^n_U) \ar[r]^-{\bar j_{2\Gamma}}
      \ar[d]_\cong& S_{\bar Y} ((\bar j_2\bar j_1)_!  \gamma_{U!}  A^n_U)
      \ar[r]^-\cong \ar[d]^\cong & S_{\bar Y}((\bar j_2\bar j_1)_!  \gamma_{U!}
      (j_2j_1)^* A^n_Y) \ar[d]^\cong \\
      S_{\bar X}( \gamma_{X!} j_{1!} j_1^* A^n_X) \ar[d]_{S_{\bar X}
        \gamma_{X!}  \epsilon_1} \ar[r]^-{\bar j_{2\Gamma}} & S_{\bar Y}(\bar
      j_{2!}  \gamma_{X!}j_{1!}j_1^* A^n_X) \ar[d]^{S_{\bar Y}\bar j_{2!}
        \gamma_{X!} \epsilon_1} & S_{\bar Y}( \gamma_{Y!}(j_2j_1)_!(j_2j_1)^*
      A^n_Y)
      \ar[dd]^{S_{\bar Y} \gamma_{Y!} (\epsilon_{21})}\\
      S_{\bar X}( \gamma_{X!} A^n_X) \ar[r]^-{\bar j_{2\Gamma}}
      & S_{\bar Y}(\bar j_{2!} \gamma_{X!} A^n_X) \ar[d]^\cong \\
      & S_{\bar Y}( \gamma_{Y!} j_{2!} j_2^* A^n_Y) \ar[r]^{S_{\bar Y}
        \gamma_{Y!} \epsilon_2} & S_{\bar Y}( \gamma_{Y!} A_Y^n) }\]
    The left-hand
    side of the diagram commutes by the naturality properties of
  $\bar j_{2\Gamma}$ (by
  Lemma~\ref{lem:all-tech}); the right-hand
  side of the diagram commutes because it commutes in $\Sh(\bar Y\times_k k^s)$
  before applying $S_{\bar Y}$.  The composition around the bottom is the map
  $F_n^!(j_2)\circ F_n^!(j_1)$, and the composition around the top is the map
  $F_n^!(j_2\circ j_1)$.  Since the diagram commutes, $F_n^!$ is functorial.
\end{proof}

We now define the functor
\begin{equation*}
    F_!^n\colon\co(\Varc_{k})\to \Chb(\Repc(\Gal(k^s/k);\Zb/\ell^n\Zb))^\op.
\end{equation*}
On objects, it is equal to $F_n^!$ (see \eqref{Wexactobn}). By definition, a
cofibration $i$ in $\Varc_{k}$ consists of a commuting square
\begin{equation*}
  \xymatrix{ Z \ar[r]^i \ar[d]_{\gamma_Z} & X \ar[d]^{\gamma_X} \\
    \bar Z \ar[r]^{\bar i} & \bar X}
\end{equation*}
where the horizontal maps are closed embeddings. Given such, we obtain a map
\begin{equation*}
  F_!^n i\colon S_{\bar X}(\gamma_{X!}A^n_X)\to S_{\bar Z}( \gamma_{Z!}A^n_Z)
\end{equation*}
using the composition of morphisms
\[\xymatrix@C=4em@R=1ex{
    S_{\bar X}(\gamma_{X!}A^n_X) \ar[r]^-{S_{\bar X}\gamma_{X!}(\eta)}
    &S_{\bar X}( \gamma_{X!}i_\ast i^\ast A^n_X) \ar[r]^\cong &
    S_{\bar X}( \gamma_{X!}i_!A^n_Z)  \\
    \qquad \ar[r]^\cong & S_{\bar X} (\bar{i}_! \gamma_{Z!}A^n_Z)
    \ar[r]^{\bar i_\Gamma^{-1}}& S_{\bar Z}(\gamma_{Z!} A^n_Z).  }\] Here,
$\eta$ is the unit of the adjunction $i^\ast\dashv i_\ast$, and the last two
isomorphisms come from the canonical identifications $i_!=i_\ast$ (because $i$
is proper), $\gamma_{X!}i_!=\bar{i}_!\gamma_{Z!}$ (by functoriality of $(-)_!$)
and $A^n_Z\cong i^\ast A^n_X$.  The last map is the inverse of the isomorphism of
Lemma~\ref{lem:all-tech}.

\begin{lemma} \label{lem:F^!}
    The functor $F_!^n$ is well-defined.
\end{lemma}
\begin{proof}
  The definition immediately implies that identities are mapped to identities,
  so we only need to check that composition is respected. A composable pair of
  complement maps in $\Varc_{k}$ consists of a commuting diagram
  \[\xymatrix{
      W \ar[r]^{i_2} \ar[d]_{\gamma_W} & Z \ar[r]^{i_1} \ar[d]_{\gamma_Z} & X
      \ar[d]^{\gamma_X} \\
      \bar W \ar[r]^{\bar i_2} & \bar Z \ar[r]^{\bar i_1} & \bar X
    }\]
    in which the horizontal maps are closed embeddings. Given this diagram, we have the following diagram:
  \[\xymatrix@C=5em{
      S_{\bar X}(\gamma_{X!}A^n_X) \ar[d]_{S_{\bar X}\gamma_{X!}\eta_1}
      \ar[rd]^{S_{\bar X}\gamma_{X!} \eta_{12}} \\
      S_{\bar X}(\gamma_{X!}i_{1*}i_1^* A^n_X) \ar[r]^-{S_{\bar
          X}\gamma_{X!}i_{1*}\eta_2} \ar[d]_\cong & S_{\bar
        X}(\gamma_{X!}(i_1i_2)_* (i_1i_2)^*A^n_X )\ar[r]^-\cong \ar[d]^\cong
      & S_{\bar
        X}(\gamma_{X!}(i_1i_2)_!A^n_W ) \ar[dd]^\cong \\
      S_{\bar X}(\bar i_{1!}\gamma_{Z!}A^n_Z) \ar[d]_{\bar
        i_{1\Gamma}} \ar[r]^-{S_{\bar X}\bar i_{1!}\gamma_{Z!}\eta_2}&
      S_{\bar X}(\bar i_{1!}\gamma_{Z!}i_{2*}i_2^*A^n_Z) \ar[d]^\cong
      &  \\
      S_{\bar Z}( \gamma_{Z!} A^n_Z) \ar[d]_{S_{\bar Z}\gamma_{Z!}\eta_2} &
      S_{\bar X}(\bar i_{1!}\gamma_{Z!}i_{2!}A^n_W) \ar[d]_{\bar
        i_{1\Gamma}^{-1}} \ar[r]^\cong
       &  S_{\bar X}((\bar{i}_1\bar{i}_2)_! \gamma_{W!} A^n_W)
      \ar[d]^{(\bar{i}_1\bar{i}_2)_\Gamma^{-1}} \\
      S_{\bar Z}(\gamma_{Z!} i_{2*}i_2^* A^n_Z) \ar[r]^\cong & S_{\bar Z} (\bar i_{2!} \gamma_{W!} A_W^n ) \ar[r]^{\bar
        i_{2\Gamma}^{-1}}
      & S_{\bar W}(\gamma_{W!}A^n_W)
    }\]
  Here, $\eta_a$ is the unit of the adjunction $i_a^* \dashv (i_a)_*$ for $a =
  1,2$ and $\eta_{12}$ is the unit for the adjuction $(i_1i_2)^* \dashv
  (i_1i_2)_*$.  The composition around the top is $F^n_!(i_1i_2)$; the
  composition around the bottom is $F^n_!(i_2)\circ F^n_!(i_1)$.   The diagram
  commutes by the naturality of $\eta_1$,$\eta_2$,$\eta_{12}$ and by
  Lemma~\ref{lem:all-tech}.  Thus $F^n_!$
  is a functor.
\end{proof}

It now remains to construct $F_n^w$.  We define it using the same formula as for
$F_n^!$.  Note that as
the proof of Lemma~\ref{lem:F^!} only used the results of
Lemma~\ref{lem:all-tech}, the proof works analogously to show that $F_n^w$ is
well-defined.

\begin{lemma}
  The collection of functors $\{(F_n^!,F^n_!,F_n^w)\}$ defines a
  $W$-exact functor
    \begin{equation*}
        F\colon \Varc_{k}\to\Chb(\RepcAR(\Gal(k^s/k);\ell))^\op.
    \end{equation*}
\end{lemma}
\begin{proof}
  The maps $F_{n+1}^!\to F_{n+1}^!\otimes_{\Z}\Z/\ell^n\Z\to F_n^!$ and
  similarly for $F^n_!$ and $F_n^w$ endow the collections $\{F_n^!\}$,
  $\{F^n_!\}$ and $\{F_n^w\}$ with the structure of projective systems. From the
  construction, for each fixed $X$, $\{F_n^!(X)\}$, $\{F^n_!(X)\}$ and
  $\{F_n^w\}$ are obtained by applying $S_{\bar X}\gamma_{X!}$ to
  A--R-$\ell$-adic systems of sheaves on $X\times_k k^s$. Therefore, by
  \cite[Theorem I.12.15]{FK}, $\{F_n^!(X)\}$, $\{F^n_!(X)\}$ and $\{F_n^w(X)\}$
  are A--R-$\ell$-adic complexes of sheaves on $k^s$, i.e. A--R-$\ell$-adic
  complexes of continuous $\Gal(k^s/k)$-modules. So, we indeed have functors
  \begin{align*}
    F^!&\colon \comp(\Varc_{k})^\op\to (\Chb(\RepcAR(\Gal(k^s/k);\ell))^\op)^\op\\
    F_!&\colon\co(\Varc_{k})\to \Chb(\RepcAR(\Gal(k^s/k);\ell))^\op \\
    F^w&\colon \w(\Varc_k) \to  \Chb(\RepcAR(\Gal(k^s/k);\ell))^\op.
  \end{align*}
  It remains to verify that $F$ is $W$-exact.  Axioms (1)-(4) hold by
  definition, so we check the remainder in turn.

  First, consider Axiom (5).  It suffices to prove it for $F_n$.  This is a
  large but straightforward diagram chase using the definitions of $F_n^!$ and
  $F^n_!$.  For those who would like to see the details, we present them in
  Appendix~\ref{app:large-diag}.

  Now we check Axiom (6): that $F$ takes subtraction sequences to exact
  sequences.  Again, it suffices to show it for $F_n$.  Given a subtraction
  sequence in $\Varc_{k}$
  \[\xymatrix{
    (Z,\bar Z) \rcofib[r]^j & (X, \bar X) & (U, \bar U) \ar[l]^{\circ}_i}\]
  we obtain a sequence in $\Sh(X\times_k k^s)$
  \begin{equation*}
    0\to j_!A^n_U \to A^n_X \to i^*A^n_Z\to 0
  \end{equation*}
  and we see that this is exact by direct inspection (i.e. by verifying
  exactness on stalks). Applying $S_{\bar X}\gamma_{X!}$, we obtain an
  exact sequence in $\Chb(\RepcAR(\Gal(k^s/k);\ell))$ (using that the Godement resolution is an exact functorial flabby resolution, see e.g. \cite[p. 129]{FK}). Consider the following
  diagram:
  \[\xymatrix{
    S_{\bar U}(\gamma_{U!}A^n_U) \ar[d]_{\bar j_\Gamma}
    \ar[rd]^{F_n^!(j)} \\
    S_{\bar X}(\gamma_{X!}j_!A^n_U) \ar[r] & S_{\bar X}(\gamma_{X!}
    A^n_X) \ar[r] \ar[rd]_{F^n_!(i)} & S_{\bar X}(\gamma_{X!}
    i^*A^n_Z) \ar[d]^{\bar i_\Gamma^{-1}} \\
    & & S_{\bar Z}(\gamma_{Z!}A^n_Z)}\]
  The exact sequence across the middle is isomorphic to the diagonal
  sequence, which is thus also exact.  Therefore $F_n$ takes subtraction
  sequences to exact sequences, as desired.

  It remains to verify Axiom (7); as before, it suffices to prove that it
  holds for $F_n$.  The proof of (7) for complement maps is identical to the
  proof of Lemma~\ref{lem:F^!}, since it is simply checking that the
  transformation defined in Lemma~\ref{lem:all-tech} respects composition.
  The proof of (7) for cofibrations is identical to the proof of (5), since
  any such commutative diagram is automatically a pullback square, and $F^w_n$
  is defined identically to $F_!^n$.
\end{proof}

\section{Nontrivial Classes in the Higher $K$-Theory of
  Varieties}
\label{sec:nontriv}

Let $\FinSet$ be the category of finite sets.  Consider the map
$\Sb \rto K(\Var_k)$ induced by the exact functor $\FinSet \rto \Var_k$ given by
$A \mapsto \coprod_A \Spec(k)$.  This induces a homomorphism
\[\Sb \to K_*(\Var_k).\] The stable homotopy groups of spheres have a rich
higher structure, but it is not clear that this structure is not annihilated
when passing to varieties.  In this section we explore this map in the cases
where $k$ is a subfield of $\Cb$ and when $k$ is a finite field.  In the case
when $k$ is a subfield of $\Cb$ we show that the image of this map is nontrivial
even above the $0$-th homotopy group; in the case when $k$ is finite we show
that $K_*(\Sb)$ is a direct summand of $K_*(\Var_k)$ and prove that this map is
not surjective when $* = 1$. Moreover, we are able to extend this result to
local and global fields which contain a place of cardinality congruent to $3$
mod $4$, as well as to subfields of $\Rb$.

Our analysis considers a particularly simple family of maps involving
permutations of varieties.  In order to construct these maps, it is useful to
have a model of the sphere spectrum $\mathbb{S}$ within SW-categories. The
following lemmas furnish this.

\begin{lemma}
  The category of finite sets $\FinSet$ is an SW-category where cofibrations are
  monomorphisms and subtraction sequences are diagrams
  $[m] \hookrightarrow [n] \hookleftarrow [n-m]$ where the images of the two
  maps are disjoint.
\end{lemma}

The category $\FinSet_+$ of finite pointed sets can be considered to be a
Waldhausen category by defining the cofibrations to be injections and the weak
equivalences to be isomorphisms.  Write $S_+$ for a set $S$ with a disjoint
basepoint added. We can define a $W$-exact functor
\[(I^!,I_!,I^w): \FinSet \rto \FinSet_+\] by defining $I(S) = S_+$ for any
finite set $S$, $I_!(i) = i_+$ for any cofibration $i$ and $I^w(i) = i_+$ for
any isomorphism $i$.  To define $I^!$ we consider an injection $j:S \rto T$,
considered as a complement map.  We can define an ``inverse'' map $T_+ \rto S_+$
by taking $t$ to $j^{-1}(t)$ when $t$ is in the image of $j$, and to the
basepoint otherwise.

\begin{prop} \label{prop:finset}
  The $W$-exact functor $(I^!,I_!,I^w):\FinSet \rto \FinSet_+$ induces an
  isomorphism
  \[K^{SW}(\FinSet) \rto K^W(\FinSet_+);\]
  here, $K^{SW}$ takes the $K$-theory of $\FinSet$ as an $SW$-category, and
  $K^W$ takes the $K$-theory of $\FinSet_+$ as a Waldhausen category.
\end{prop}
\begin{proof}
  It suffices to check that $I$ induces an isomorphism of $n$-simplicial
  categories
  $\widetilde{S}^{(n)}_\bullet \FinSet \cong S_\bullet^{(n)} \FinSet_+$.  To
  check this, it suffices to check that $I$ induces an isomorphism of categories
  at each simplicial level; this follows because $I$ bijectively renames the
  terms in each object of $\widetilde{S}^{(n)}_{i_1\cdots i_n}\FinSet$.
\end{proof}

We now define the family of maps we use.
\begin{defn}
  For each $X \in \mc{V}_k$, define
  \[
    \sigma_X \colon \FinSet \to \Var_k
  \]
  to be the exact functor of $SW$-categories $\FinSet \to \mc{V}_k$ induced by
  \[
    [n] \longmapsto \underbrace{X \amalg \cdots \amalg X}_{n \ \text{times}}.
  \]
  Thus we get a family of maps
  \[\pi_*\sigma_X: \pi_*(\Sb) \rto K_*(\Var_k).\]
\end{defn}

\subsection{Subfields of $\Cb$}

We appeal to some facts from Adams \cite{adams_JX_4} and Quillen's letter to Milnor \cite{quillen_letter}. Recall that there is a homomorphism from the stable homotopy of the stable orthogonal group to the stable homotopy groups of spheres: $J\colon \pi^s_\ast \mathbf{O} \to \pi_\ast \mathbb{S}$. This is constructed as follows. Since every element of $O(n)$ defines a map $\mathbf{R}^n \to \mathbf{R}^n$, by one point compactification, it defines a point map $S^n \to S^n$. Thus, there is a map of spaces $O(n) \to \operatorname{Map}_\ast (S^n, S^n) \cong \Omega^n S^n$. Taking homotopy groups and stabilizing yields the $J$-homomorphism. The following is due to Adams

\begin{thm}\cite[Thm. 1.5]{adams_JX_4} \label{thm:adams}
The map $J: \pi_{4s-1} \mathbf{O} \to \pi_{4s-1} \mathbb{S}$ exhibits $\pi_{4s-1}\mathbf{O}$ as a direct summand of $\pi_{4s-1} \mathbb{S}$ and the image is
\[
J(\pi_{4s-1}(\mathbf{O})) \cong C_{d_s} \qquad d_s = \text{denominator} \left(\frac{B_s}{4s}\right)
\]
where $B_s$ is the $s$th Bernoulli number, and $C_{d_s}$ denotes the cyclic
group of order $d_s$.
\end{thm}

There is also a map $p_i\colon\pi_i \mathbb{S} \to K_i (\Z)$ induced by the
inclusion $B \Sigma_\infty \to BGL(\Z)$. For our purposes, it is more
useful to regard it as the map induced by the exact functor of Waldhausen
categories
\[
  P\colon \FinSet_+ \to \Modf(\Z) \qquad [n]_+
  \mapsto \Z \oplus \cdots \oplus \Z
\]
Quillen's letter to Milnor gives the following.
\begin{thm}\cite{quillen_letter} \label{thm:quillenletter}
The composite map
\[
\pi_{4s-1}(\mathbf{O}) \to^J \pi_{4s-1}( \mathbb{S})   \to^{p_{4s-1}} K_{4s-1} (\Z)
\]
is injective.
\end{thm}

We can use these results to identify nontrivial elements in $K_{2s-1}(\Var_k)$.

\begin{thm}
  Let $k$ be a subfield of $\mathbf{C}$. There are infinitely many non-trivial
  homotopy groups of $K_\ast (\mc{V}_k)$. In particular, for all $s > 0$,
  $K_{4s-1}(\mc{V}_k)$ is non-trivial.
\end{thm}
\begin{proof}
  Fix $s > 0$, and let $X \in \mc{V}_k$ be a projective variety such that its
  compactly-generated Euler characteristic $\chi_c(X)$ is relatively prime to
  $d_s$.  (Note that this is always possible, as for example
  $\chi_c (\mathbf{C} P^n) = n$ for all $n$.)  Let $C^*_c(X)$ be the
  compactly-supported singular chains on the complex points of $X$.
  Consider the diagram
  \[\xymatrix@C=4em{
      \Var_k & \Varc_k \ar[l]_U \ar@{-->}[r]  & \Chb(\Z) & \ar[l]_{\cdot[0]}\Modf(\Z) \\
      & \FinSet \ar[lu]^{\sigma_X} \ar[u]|{\sigma_{(X,X)}} \ar@{-->}[r]^I &
      \FinSet_+\ar[r]^P & \Modf(\Z) \ar[lu]|{\otimes C^*_c(X)}} \] where the
  left half of the diagram consists of $SW$-categories and the right half
  consists of Waldhausen categories; both dashed arrows are $W$-exact functors.
  This diagram commutes, in the sense that the left-hand triangle commutes as a
  triangle of exact functors, and the middle pentagon commutes as a diagram of
  $W$-exact functors (as $W$-exact functors can be pre/postcomposed with exact
  functors).  The bottom dashed arrow is the $W$-exact functor $I$ constructed
  above Proposition~\ref{prop:finset}.  Applying $\pi_0$ and noting that the two
  solid arrows across the top become isomorphisms we recover the motivic measure
  constructed in Corollary~\ref{cor:Reuler}.  Upon applying $K$-theory we get a
  diagram
  \[\xymatrix@C=4em{
      K_{4s-1}(\Var_k) & K_{4s-1}(\Varc_k) \ar[l]_{\cong} \ar[r] &
      K_{4s-1}(\Chb(\Z))
      \ar[r]^-{\chi} &  K_{4s-1}(\Z) \\
      \pi_{4s-1}(\mathbf{O}) \ar[r]^J & \pi_{4s-1}\Sb \ar[lu]|{\sigma_X}
      \ar[r]^\cong \ar[u]|{\sigma_{(X,X)}} & \pi_{4s-1}\Sb \ar[r]^{p_{4s-1}} &
      K_{4s-1}(\Z) \ar[lu] \ar[u]_{\cdot \chi_c(X)}} \] where the isomorphism
  induced by $I$ is the identity.  By
  Theorems~\ref{thm:adams} and \ref{thm:quillenletter}, the composition across
  the bottom $\pi_{4s-1}(\mathbf{O}) \rto K_{4s-1}(\Z)$ is injective; since
  $\chi_c(X)$ is relatively prime to $d_s$, the composition
  $\pi_{4s-1}(\mathbf{O}) \rto K_{4s-1}(\Z)$ across the bottom and to the
  upper-right is also injective.  Since this factors through $K_{4s-1}(\Var_k)$
  the theorem follows.
\end{proof}

\begin{remark}
  In the above proof we are not really using anything interesting about $X$; in
  particular, we could take $X$ to be $\Spec k$ and the proof still goes
  through.  This makes sense, as the image of $J$ sees nontrivial elements of
  $K$-theory which can be described in terms of permutations, which are
  automorphisms of $0$-dimensional varieties.

  However, we believe that the more complex proof is useful due to the
  possibility of generalization.  It should be possible to use the same
  construction in the proof to show that there are more interesting nontrivial
  elements in higher homotopy groups of $K(\Var_k)$ by exploiting the structure
  of the cohomology of $C_c^*(X)$.  For example, if the measure can be enriched
  to land in the $K$-theory of mixed Hodge structures then by selecting an $X$
  with a nontrivial mixed Hodge structure the above proof would detect an
  element which is \emph{not} in the image of the map $\Sb \rto K_*(\Var)$
  induced by the inclusion of $0$-dimensional varieties.
\end{remark}

\subsection{Finite Fields}
When $k$ is finite the functor
\[
\begin{array}{rclcrcl}
  \Var_k &\to &\FinSet\\
  X & \longmapsto& X(k)
\end{array}
\]
gives a splitting of the $K$-theory spectrum $K(\Var_k)$ as
$K(\Var_k)\simeq \Sb\vee \tilde{K}(\Var_k)$ (and thus of each homotopy group).  A
priori it may be the case that all higher homotopy groups of $K(\Var_k)$ are in
the image of the homotopy groups of $\Sb$.

We show that this is not the case by using $\Zeta$ to identify a nonzero element
in $\tilde K_1(\Var_k)$ for $k=\F_q$ (with $q\equiv 3$ mod 4).  For such $k$, we
construct a surjective homomorphism $h_2: K_1(\Var_{k}) \rto \Z/2$ such that the
composition $\pi_1(\Sb) \rto K_1(\Var_{k}) \rto^{h_2} \Z/2$ is trivial.  The map
$h_2$ is defined as follows.  Let $\star$ be the operation defined in \cite[\S
8]{milnor}.  Grayson \cite{Gr79} has shown that given a pair of commuting
automorphisms $f,g$ on a $\Qb_\ell$-vector space $V$, the map
\begin{equation*}
  (f,g) \mapsto f^{-1} \star g,
\end{equation*}
induces a homomorphism $s_\ell:K_1(\Aut(\Qb_\ell))\to
K_2(\Qb_\ell)$. Moore's Theorem (see \cite[Appendix]{milnor}, or \cite[Theorem
57]{WeibelLocal} and its proof) shows that the mod-$\ell$ Hilbert symbol gives a
split surjection $(-,-)_\ell\colon K_2(\Qb_\ell)\onto \mu(\Qb_\ell)$ onto the
roots of unity in $\Qb_\ell$ with kernel an uncountable uniquely divisible
abelian group $U_2(\Qb_\ell)$. For $q$ odd, we then define $h_2$ to be the
composition
\begin{equation*}
    K_1(\Var_{\F_q}) \rto^{\pi_1(\Frob_*\circ\Zeta)} K_1(\Aut(\Qb_2))
\rto^{s_2}
K_2(\Qb_2) \rto^{(-,-)_2} \Z/2,
\end{equation*}
where $\Frob_*$ is the map defined in Corollary ~\ref{cor:gZeta} for $g=\Frob$
and $(-,-)_2$ is the $2$-adic Hilbert symbol.

Fix a variety $X$, and consider the following diagram:
\begin{equation} \label{eq:AutZ2}
  \xymatrix@C=3em{
    & K_1(\Varc_k) \ar[r]^-{h_2} \ar[d]_{K_1(U)}&  \Z/2\\
    \pi_1\Sb \ar[r]^{\pi_1\sigma_X} \ar@{.>}@/^{1.5em}/[ru]^-{\pi_1\sigma_{(X,X)}} & K_1(\Var_k)
  }
\end{equation}
If $X$ is proper then $\sigma_{(X,X)}$ exists, and the diagram commutes with the
dotted arrow added.

\begin{theorem} \label{thm:non-EM-finite} When $k=\F_q$ with $q \equiv 3 \pmod 4$
  the map $h_2$ defined above detects a nontrivial class in
  $\tilde{K}_1(\Var_k)$.
\end{theorem}

\begin{proof}
  Let $\eta\in \pi_1(\mathbb{S})$ be the nonzero element.  We show that $h_2$ is
  nonzero but contains $(\pi_*\sigma_{\Spec k})(\eta)$ in its kernel.  This
  shows that $\tilde{K}_1(\Var_{k}) \neq 0$.

  The difficult part of computing $h_2$ is computing the image of $s_2$ in
  $K_2(\Q_2)$.  This is done using the $\star$ map, discussed in detail in
  \cite[Chapter 8]{milnor}.  As we are using many of the results in the section
  we will not cite them individually; we simply note that the important
  components are the definition of the general $\star$ on page 57, the proof
  that $\star$ is bilinear in Theorem 8.8 and Lemma 8.3.

  Let $\tau: \{1,2\} \rto \{1,2\}$ be the transposition of two points.

  The cohomology of a point is concentrated in degree $0$, with Frobenius acting
  trivially; thus $\tau^* \star Frob_q$ is trivial, and $h_2(\pi_1\sigma_{\Spec
    k}(t)) = 1$ in $\Z/2$.

  Now let $X=\Pb^1$, and consider the automorphism $\alpha:x \mapsto 1/x$ acting
  on $\Pb^1$.  For $k=\F_q$, we can write
  $\alpha \circlearrowleft (H^\ast_{et,c}(\Pb^1|_{\Fqbar}; \Qb_2),
  \Frob^{\Pb^1}_q)$ as a direct sum (in $\Aut(\Qb_2)$)
  \begin{equation*}
    (1 \circlearrowleft \Qb_2(0)) \oplus (-1 \circlearrowleft
    \Qb_2(-1)).
  \end{equation*}
  $Frob_q$ acts on the first summand by $1$, and on the second by $q$; $\alpha$
  acts on the first summand by $1$ and on the second by $-1$.  Since both
  summands are in even degrees we can disregard degree in our computations; thus
  we have a module $\Q_2 \oplus \Q_2$ with the automorphism $(1\oplus q)$ as our
  object of $\Aut(\Q_2)$; we have the extra automorphism $(1\oplus -1)$ acting
  on this.  Since $\star$ is bilinear, to compute the image of this under $s_2$
  it suffices to compute $q \star -1$.  By definition
  \[q \star -1 =
    \left(\begin{array}{ccc}
            q\\ & q^{-1} \\ & & 1
          \end{array}
        \right)
        \star
        \left(\begin{array}{ccc}
                -1 \\ & 1 \\ & & -1
              \end{array}
            \right)
            = \{q,-1\}\{q^{-1},1\}\{1,-1\} = \{q,-1\}\in K_2(\Q_2).\] The
          image of this under $h_2$ is $(q,-1)_2$, which, when $q\equiv 3$ mod $4$, is
          $-1$.  This gives the desired element in $\tilde{K}_1$.
        \end{proof}

  The last computation in the proof above shows that in general, for $f$ and
  $g$ automorphisms of a one-dimensional module, $f\star g = \{f,g\}$.

\subsection{Elements in $K_1$ for other fields}

Suppose that $k$ is a global or local fields with a place of cardinality
equivalent to $3$ mod $4$, or $k\subset \Rb$.  The results from the previous
section hold with almost identical proofs, with slightly different definitions.

For $k$ a global or local field with a place of cardinality equivalent to $3$
mod $4$, we pick a Frobenius element $\phi$ for this place, and define $h_2$ to
be the composition
\begin{equation*}
    K_1(\Var_{k}) \rto^{\pi_1(\phi_*\circ\Zeta)} K_1(\Aut(\Qb_2))
\rto^{\sigma_2}
K_2(\Qb_2) \rto^{(-,-)_2} \Z/2,
\end{equation*}
Similarly for $k\subset\Rb$, we define $h_2$ to be the composition
\begin{equation*}
    K_1(\Var_{k}) \rto^{\pi_1((\bar \cdot)_*\circ\Zeta)} K_1(\Aut(\Qb_2))
\rto^{\sigma_2}
K_2(\Qb_2) \rto^{(-,-)_2} \Z/2,
\end{equation*}
where $\bar\cdot$ denotes complex conjugation and $(\bar \cdot)_\ast$ is the map defined in
Corollary ~\ref{cor:gZeta} for $g=\bar \cdot$.

\begin{theorem} \label{thm:non-EM-others}
  For $k$ a global or local field with a place of cardinality equivalent to $3$
  mod $4$, or for $k\subset \Rb$, $\tilde K(\Var_k)$ has nontrivial higher
  homotopy groups.
\end{theorem}

\begin{remark}
  We expect that $\tilde{K}(\Var_{k})$ has nontrivial higher homotopy groups in
  general, however the particular class that we use $h_2$ to detect requires the
  assumptions on $k$. Using the methods of this paper, it should be possible to
  find a different example that gives a nontrivial class for all odd $q$, all
  global and local fields with a place of odd cardinality, and all
  non-algebraically closed subfields of $\Cb$. For even $q$, one would need the
  $p$-adic (rather than the $\ell$-adic) analogue of the derived zeta function
  to employ the present approach.
\end{remark}

\begin{proof}[Proof of Theorem~\ref{thm:non-EM-others}]
  This proof works the same way as the proof of
  Theorem~\ref{thm:non-EM-finite}.

  If $k$ is a global or local field with a place of cardinality $3$ mod $4$,
  then, for any Frobenius element $\phi$ for this place, the action of $\phi$ on
  the \'etale cohomology of $\Pb^1$ factors through the action on the special
  fiber.  In particular, the above computation similarly shows that the map
  $h_2$ takes the class $\alpha\circlearrowleft \mathbb{P}^1$ to $-1$, and thus
  gives the desired element in $\tilde{K}_1(\Var_k)$.

  For $k\subset\Rb$, we can consider the same $X$. We write
  $\pi_1\alpha\circlearrowleft (H^\ast_{et,c}(X|_{\Cb}; \Qb_2),
  \bar\cdot)$ as a direct sum
  \begin{equation*}
    (1 \circlearrowleft (\Qb_2,1))
    \oplus  (-1 \circlearrowleft
    (\Qb_2,-1)).
  \end{equation*}
  The map $s_2$ sends everything with the identity acting on it to the unit, so
  the image of this under $h_2$ is $(-1,-1)_2=-1$.  This gives the desired
  element in $\tilde{K}_1$.
\end{proof}

Theorem~\ref{thm:calc} is a direct consequence of
Theorems~\ref{thm:non-EM-finite} and \ref{thm:non-EM-others}.

\begin{remark}
  It should be possible to do a more powerful analysis on $K_1$ by exploiting
  the rich structure of automorphism groups of varieties and applying
  Proposition~\ref{prop:xiX}.
\end{remark}

\section{Questions for Future Work}\label{sec:prospects}

\subsection*{Indecomposable Elements in $K(\Var_k)$.}
Theorems \ref{thm:non-EM-finite} and \ref{thm:non-EM-others} establish that
there are non-trivial classes in the higher $K$-theory of varieties that do not
come from the sphere spectrum.  However, one could ask a more refined question:
since $K(\Var_k)$ is an $E_\infty$-ring spectrum, its homotopy groups
$K_*(\Var_k)$ form a ring.  We therefore have a ready supply of elements of
$K_*(\Var_k)$: those in the image of the multiplication
\[\beta: K_0(\Var_{k}) \otimes \pi_*(\Sb) \xrightarrow{1\otimes \sigma_{\Spec k}} K_0(\Var_{k}) \otimes
K_*(\Var_{k}) \to K_*(\Var_{k}).\] We call such elements
\emph{decomposable}.  A priori, it may be the case that this map is
surjective, and that therefore all higher homotopy groups of $K(\Var_{k})$
are decomposable.  The example constructed in Section \ref{sec:nontriv} is decomposable, since this can
just be written as $\eta \cdot [\Pb^1]$.

\begin{question}{\bf Indecomposable elements.}\mbox{}
    Do there exist indecomposable elements in $K_*(\Var_k)$?
  \end{question}

  As we explain in Remark \ref{rmk:noth2} below, we do not expect the map $h_2$
  to be able to distinguish decomposable from non-decomposable
  elements. Instead, we hope that by expanding the collection of derived motivic
  measures and employing Proposition~\ref{prop:xiX} judiciously, a suitable
  invariant could be found.

\begin{remark}\label{rmk:noth2}
  We expect the derived $\ell$-adic zeta function to factor through the $K$-theory of the Morel-Voevodsky category of motivic spectra, via the construction of compactly supported integral motivic cohomology.  This suggests that the invariant $h_2\colon K_1(\Var_k)\to K_2(\Qb_2)$ in Section \ref{sec:nontriv}
  factors through the composition
    \begin{equation*}
        K_1(\Aut(\Z))\to K_2(\Z)\to K_2(\Qb)\to K_2(\Qb_2).
     \end{equation*}
    We highlight three implications of this expected factoring:
    \begin{enumerate}
        \item It underscores the importance of the 2-adic Hilbert symbol, as opposed to the $\ell$-adic Hilbert symbol for $\ell\neq 2$.  Indeed, by Tate's computation of $K_2(\Qb)$ (see e.g. \cite[Theorem 11.6]{milnor}), the map $K_2(\Z)\to K_2(\Qb)$ is split injective, with the splitting given by the 2-adic Hilbert symbol.  In particular, the Hilbert symbols $(-,-)_\ell$ for $\ell\neq 2$ identically vanish on $K_2(\Z)$. Further, no classes outside the summand $\mu(\Qb_2)\subset K_2(\Qb_2)$ are in the image of $K_2(\Z)$.
        \item It suggests that we should not expect the map $h_2$ to be able to distinguish indecomposable elements in $\tilde{K}_1(\Var_k)$. Indeed, Milnor's computation \cite[Corollary 10.2]{milnor} shows that the map $K_1(\Aut(\Z))\to K_2(\Z)$ is surjective and the nontrivial class in $K_2(\Z)=\Z/2\Z$ is mapped onto by decomposable classes.
        \item It suggests that the higher invariants of the derived zeta functions should be in some sense independent of $\ell$. It would be fruitful to understand this more precisely!
    \end{enumerate}
\end{remark}

\subsection*{Other Derived Motivic Measures}
The recipe in this paper should work to construct derived motivic measures for
other cohomological invariants. We took $\ell$-adic cohomology as the basis for
our derived zeta function.  One would like analogous maps for the other Weil
cohomology theories.
\begin{problem}{\bf Derived $p$-adic zeta functions.}\mbox{}
    Let $k$ be a perfect field of characteristic $p$ with Witt vectors $W(k)$.  Construct a map of $K$-theory spectra
    \begin{equation*}
        K(\Var_k)\to K(\Aut(W(k)))
    \end{equation*}
    which lifts the function sending a variety $X$ to $H^\ast_{\rig,c}(X/W(k))$ to its compactly supported rigid cohomology (with constant coefficients) acted on by the Frobenius automorphism.
\end{problem}
We expect that the construction should parallel that in Section \ref{sec:approach}, with the category $\Varc$ replaced by a category of varieties $X$ equipped with a choice of compactification $X\into \bar{X}$, and a choice of map of admissible triples $(X,Y,\mc{Y})\to (\bar{X},\bar{Y},\bar{\mc{Y}})$ extending $X\into\bar{X}$ as in \cite[Section 3]{Be}, and with rigid cohomology replacing the $\ell$-adic constructions. Note that, Tsuzuki's finiteness theorem \cite[Theorem 5.1.1]{Ts} plays an essential role in defining the $W$-exact functor.

\begin{problem}{\bf Derived Serre Polynomial.}\mbox{}
    Let $k$ be a field of characteristic 0. Construct a map of $K$-theory spectra taking values in the $K$-theory of integral mixed Hodge structures
    \begin{equation*}
        K(\Var_k)\to K(\MHS_\Z)
    \end{equation*}
    which lifts the function sending a variety $X$ to $H^\ast_c(X(\Cb);\Z)$ with its canonical mixed Hodge structure.
\end{problem}
We expect that the construction should parallel that in Section \ref{sec:approach}, with the category $\Varc$ replaced by a category of varieties $X$ equipped with a choice of compactification $X\into \tilde{X}$, and a choice of cubical hyperresolution $\tilde{X}_\bullet\to \tilde{X}$ of the pair $(\tilde{X},\tilde{X}-X)$ (see e.g. \cite[Chapter 5]{PS}), and with logarithmic differential forms in lieu of the $\ell$-adic constructions.

The framework of motives suggests that the derived $\ell$-adic zeta function, along with the two maps described above, should factor through a derived motivic measure built from motivic cohomology.
\begin{problem}{\bf Derived Gillet--Soul\'e.}\mbox{}
    Let $k$ be a field admitting resolution of singularities.  Construct a map of $K$-theory spectra taking values in the $K$-theory of integral Chow motives over $k$
    \begin{equation*}
            K(\Var_k)\to K(\M_k)
    \end{equation*}
    which lifts the motivic measure of Gillet--Soul\'e \cite{GS} sending a $k$-variety $X$ to its compactly supported integral Chow motive. Prove that the derived $\ell$-adic zeta function and the derived Serre polynomial factor through this map.
\end{problem}
We expect that the replacement of $\Varc$ should be the same as for the Serre polynomial.

For general fields $k$, one might expect to have a motivic measure based on integral Voevodsky motives through which all of the above maps factor.

Moving further away from cohomological invariants, one of the richest motivic measures is Kapranov's motivic zeta function
\begin{align*}
    K_0(\Var_k)&\to W(K_0(\Var_k))\\
    [X]&\mapsto \sum_{i=0}^\infty [\Sym^i(X)]t^i.
\end{align*}

\begin{question}{\bf Derived motivic zeta function.}\mbox{}
    Does Kapranov's motivic zeta function lift to a map of $K$-theory spectra?
\end{question}
For motivation, recall that Weil's realization that the classical zeta function of a variety over a finite field can be obtained cohomologically provided a robust strategy for proving that the zeta function of such varieties is rational.  Similarly, a lift of Kapranov's motivic zeta function to a map of $K$-theory spectra might be expected to go a long way toward proving that the motivic zeta function is rational, in an appropriate sense.

Recall, however, that purely as a map out of $K_0(\Var_{\Cb})$, Kapranov's motivic zeta function is \emph{not} rational.  This was proven by Larsen and Lunts \cite{LL}, and the key tool in their proof was a motivic measure
\begin{equation*}
    \mu_{LL}\colon K_0(\Var_{\Cb})\to \Z[\SB_{\Cb}]
\end{equation*}
taking values in the free abelian group on stable birational equivalences classes of complex varieties.  Note that, since $\Pb^1\sim_{\SB} \ast$, Larsen and Lunts' measure takes the class of the affine line to $0$. In particular, it still may be the case that Kapranov's motivic zeta is rational after inverting the affine line, or performing some other modification of $K_0(\Var_k)$ (cf. \cite{LL2}). This underpins the ``in the appropriate sense'' above.

\begin{question}{\bf Derived Larsen--Lunts.}\mbox{}
  Does the Larsen--Lunts measure lift to a map of $K$-theory spectra? Using the formalism of assemblers, the third-named author was able to accomplish this \cite{zakharevich_annihilator}. However, it would be desirable to have a direct construction of this motivic measure. For this, one would need an SW-category which naturally encodes stable birational equivalence of projective varieties.
\end{question}

\appendix

\section{Functorial factorization of weak cofibrations} \label{app:FFWC}

In this appendix, we define and verify in the cases of interest the technical condition ``functorial factorization of weak cofibrations'' defined in \cite{BM-ku}. The condition is meant to be a weakening of Waldhausen's cylinder functors \cite[Sect. 1.6]{waldhausen}.

\begin{defn}
Let $[1]$ denote the ordered set $0 < 1$ considered as a category, and let $[2]$
the ordered set $0 < 1 < 2$, also considered as a category. A \emph{functorial
  factorization} is a functor $\varphi:\Fun([1], \mc{C}) \to \Fun([2], \mc{C})$ such
that $(d^1)^*\circ \varphi = 1_{\Fun([1],\mc{C})}$.  Here, $d^1:[1] \rto [2]$
takes $0$ to $0$ and $1$ to $2$.
\end{defn}

\begin{defn}
  Let $\mc{C}$ be a Waldhausen category. A \emph{weak equivalence} between
  morphisms $f:A \rto B$ and $g:C \rto D$ is a diagram
  \[\xymatrix{
      A \ar[r]^f \ar[d]_\sim & B \ar[d]^\sim \\
      C \ar[r]^g & D
    } \]
  A \emph{weak cofibration} is a map
  $A \to B$ that admits a zig-zig of weak equivalences to a cofibration
  $A' \hookrightarrow B'$.
\end{defn}

\begin{defn}\cite[Definition 2.2]{BM-ku}
  Let $\mc{C}$ be a Waldhausen category. Write $\Fun^{wc}([1],\C)$ for the full
  subcategory of $\Fun([1],\C)$ consisting of those functors whose image is a weak
  cofibration.  Let $\Fun^{c,w}([2],\C)$ be the full subcategory of
  $\Fun([2],\C)$ consisting of those diagrams which are a cofibration followed
  by a weak equivalence.  A \emph{functorial factorization of weak cofibrations}
  is a functor $\varphi:\Fun^{wc}([1],\C) \rto \Fun^{c,w}([2],\C)$ such that
  $(f^1)^*\circ \varphi = 1_{\Fun^{wc}([1],\C)}$.
\end{defn}

Let $\C$ be a cofibrantly generated model category.  We thus have a functorial
factorization
\[\Fun([1],\C) \rto \Fun^{c,w}([2],\C),\]
given by the functorial factorization of morphisms into a cofibration followed
by an acyclic fibration.  If a Waldhausen category arises as a subcategory of a
model category, we can often leverage this factorization to obtain functorial
factorizations inside the Waldhausen category.  The main problem is that objects
in a Waldhausen category need to be small (in some sense), whereas functorial
factorizations often produce very large objects.  However, in
Lemma~\ref{lem:chains} we showed that being \emph{homologically} small is
sufficient; when weak equivalences are quasi-isomorphisms this is therefore
sufficient.

To produce the functorial factorizations that we need we appeal to a theorem of
Beke \cite{beke} which produces model category structures on categories of
chain complexes. A version of this theorem is also proved in \cite[Theorem 2.2]{hovey-str}.

\begin{theorem}\cite[Proposition 3.13]{beke}
  Let $\mc{A}$ be a Grothendieck abelian category, i.e. $\mc{A}$ admits a generator, has small colimits, and filtered colimits commute with finite limits. Then $\operatorname{Ch}(\mc{A})$ admits a cofibrantly generated model structure where
  \begin{itemize}
  \item weak equivalences are quasi-isomorphisms
  \item cofibrations are injections
  \item fibrations have the right lifting property with respect to trivial cofibrations
  \end{itemize}
\end{theorem}

This entitles us to the following two theorems.

\begin{theorem}
  The category $\Chb(\Repc(G_k);\ell)$ admits functorial factorization of weak equivalences.
\end{theorem}
\begin{proof}
  As noted in Proposition \ref{prop:contGal}, the category
  $\Repc_{cts}(G_k;\ell)$ is equivalent to the category
  $\operatorname{Sh}^{e t}(\Spec(k); \ell)$. This is the category of sheaves of
  an abelian group on a ringed site, as such, it is a Grothendieck category.  By
  Beke's Theorem, $\operatorname{Ch}(\Repc_{cts}(G_k;\ell))$ admits a
  cofibrantly generated structure. Thus, all morphisms in
  $\operatorname{Ch}(\Repc_{cts}(G_k; \ell)$ have functorial factorizations.
  Given a weak cofibration $A \to B$ factor it as $A \hookrightarrow C \to B$;
  it remains to show that $C \in \Chb(\Repc_{cts}(G_k;\ell))$.  However, since
  $C \to B$ is a weak equivalence and $B$ is homologically bounded, $C$ must be
  as well.
\end{proof}

\begin{theorem}
The category $\Chb(R)$ admits functorial factorization of weak equivalences.
\end{theorem}
\begin{proof}
  The category $\Mod_R$ is a Grothendieck abelian category. As such,
  $\Ch(\Mod_R)$ has a cofibrantly generated injective model structure. Thus, the
  category of all maps $X \to Y$ in $\Ch(\Mod_R)$ has a functorial factorization
  of the required form. Restricting to $\Chb(\Mod_R)$, as in the previous proof,
  we see that $\Chb(\Mod_R)$ does as well.
\end{proof}



\section{The proof of Axiom (5)} \label{app:large-diag}

Axiom (5) states the following.  Suppose that we are given a cartesian diagram
\[\xymatrix{ (X, \bar X) \compar[r]^i \rcofib[d]_f & (Z,\bar Z) \rcofib[d]^g \\
    (Y, \bar Y) \compar[r]^j & (W,\bar W).}\]
We must show that
\[ \xymatrix@C=4em{
    F_n(X,\bar X) \ar[r]^{F^!_n(i)}  & F_n(Z, \bar Z)  \\
    F_n(Y, \bar Y) \ar[r]^{F^!_n(j)} \ar[u]^{F_n^!(f)}
    & F_n(W, \bar W)
    \ar[u]_{F_!^n(g)}
  }\]
commutes.  Note that since $f$ and $g$ are both closed,
we know that on sheaves $f_* = f_!$ and $g_* = g_!$.


First, note that the following diagram (in $\Chb(\RepcAR(\Gal(k^s/k));\ell)$) commutes.
\[\xymatrix@C=3em{
  S_{\bar Y}\gamma_{Y!} A^n_Y \ar[d]_{S_{\bar Y}\gamma_{Y!}\eta}
  \ar[r]^{\bar j_\Gamma} &
  S_{\bar W}\gamma_{W!} j_!A^n_Y \ar[d]^{S_{\bar W}\gamma_{W!}j_! \eta} \ar[r]^\cong & S_{\bar W}\gamma_{W!}j_!j^*A^n_W \ar[d]^{S_{\bar W}\gamma_{W!}j_! \eta_{j^*}} \\
  S_{\bar Y}\gamma_{Y!} f_*f^*A^n_Y \ar[r]^{\bar j_{\Gamma}}
  \ar[dd]_\cong &
  S_{\bar W}\gamma_{W!}j_!f_*f^*A^n_Y  \ar[dd]^\cong \ar[r]^\cong & S_{\bar W}\gamma_{W!}j_!f_*f^*j^*A^n_W
  \ar[d]^\cong
  \\
  & & S_{\bar W}\gamma_{W!}j_!f_*i^*g^*A^n_W \ar[d]^\cong \\
  S_{\bar Y}\gamma_{Y!} f_! A^n_X  \ar[r]^{\bar j_{\Gamma}}
  \ar[dd]_{\bar f_\Gamma^{-1}} \ar@{}[rdd]|{*}
  & S_{\bar W}\gamma_{W!}j_!f_! A^n_X \ar[d]^\cong \ar[r]^\cong & S_{\bar W}\gamma_{W!}j_!f_! i^*A^n_Z \ar[d]^\cong \\
  & S_{\bar W}\gamma_{W!} g_!i_! A^n_X \ar[d]^{\bar g_\Gamma^{-1}} \ar[r]^\cong
  & S_{\bar W}\gamma_{W!}g_!i_!i^*A^n_Z \ar[r]^-{S_{\bar W}\gamma_{W!}g_!\epsilon}
  \ar[d]^{\bar g_\Gamma^{-1}} & S_{\bar W}\gamma_{W!} g_!A^n_Z \ar[d]^{\bar g_\Gamma^{-1}} \\
  S_{\bar X}\gamma_{X!} A^n_X \ar[r]^{\bar i_\Gamma} & S_{\bar Z}\gamma_{Z!} i_!A^n_X \ar[r]^\cong &
  S_{\bar Z}\gamma_{Z!} i_!i^*A^n_Z
  \ar[r]^{S_{\bar Z}\gamma_{Z!}\epsilon}&
  S_{\bar Z}\gamma_{Z!} A^n_Z
}
\]
The pentagon marked $*$ commutes because of the naturality of $(-)_\Gamma$. The
composition around the bottom is $F_n^{!}(i)\circ F^n_!(f)$.

Now consider the following diagram in $\Sh(W\times_k k^s)$:
\[\xymatrix@C=3.3em{
    j_!j^*A^n_W \ar[r]^= \ar[d]_{j_!\eta_{j^*}}  & j_!j^*A^n_W
  \ar[r]^-\epsilon \ar[d]^{j_!j^*\eta} & A^n_W \ar[d]^{\eta} \\
  j_!f_*f^*j^* A^n_W
  \ar[d]_\cong \ar@{}[ru]|\star
   & j_!j^*g_*g^*A^n_W
  \ar[r]^-{\epsilon} \ar[dd]^{\cong} &
  g_*g^*A^n_W \ar[dd]^{\cong}\\
  j_!f_*i^*g^*A^n_W \ar[ur]|{j_!\alpha_{g^*}} \ar[d]_\cong \\
   j_!f_!i^*A_Z^n \ar[d]_\cong \ar[r]^{j_!\alpha}&   j_!j^*g_!A^n_Z  \ar[r]^\epsilon  & g_!A^n_Z \ar[d]^\cong\\
  g_!i_!i^* A^n_Z \ar[rr]^-{g_!\epsilon} \ar@{}[rru]|\dagger && g_!A^n_Z
}
\]
Here, it is important to keep in mind that since the original diagram of
varieties is Cartesian (and since $f_*=f_!$ and $g_* = g_!$), there is a natural isomorphism
\[\alpha: f_!i^*\cong f_*i^*  \Longrightarrow j^*g_*\cong j^*g_!.\]
The only two parts of this diagram which do not commute by definition are the two
pentagons marked $\star$ and $\dagger$.  To see that these commute, it suffices
to check that the two diagrams
\[\xymatrix@C=4em{
    f_*f^*j^* A^n_W \ar[d]_\cong  & j^* A^n_W  \ar[l]_-{\eta(f)_{j^*}} \ar[d]^{j^*\eta(g)}\\
    f_*i^*g^*A^n_W \ar[r]^{\alpha_{g^*}} & j^*g_*g^*A^n_W
  }\]
and
\[\xymatrix@C=4em{
    j_!f_!i^*A^n_Z \ar[r]^{j_!\alpha} \ar[d]_\cong &
    j_!j^*g_!A^n_Z \ar[d]^{\epsilon(j)} \\
    g_!i_!i^*A^n_Z \ar[r]^{\epsilon(i)} & g_!A^n_Z
  }\]
commute. That these commute follows directly from the definition of the base change homomorphism (see e.g. \cite[p. 60]{FK}). More conceptually, base change (and proper push forward) preserves the base change isomorphism.

After applying $S_{\bar W}\gamma_{W!}$ to this diagram, it fits into the rectangle in the
upper-right in the above diagram; then the composition around the top is $F^n_!(g)\circ F_n^!(j)$.  This proves Axiom (5).

\bibliographystyle{amsalpha}
\bibliography{CWZ}

\end{document}